\documentclass[a4paper]{amsart}
\usepackage{textcomp, amsmath, amssymb, amsthm}
\usepackage{graphicx}
\usepackage{enumerate}
\usepackage{xcolor}

\usepackage{hyperref}
\usepackage{ifthen}
\usepackage{pict2e}
\usepackage{xargs}
\usepackage{xspace}
\usepackage{pgf,tikz}
\usepackage{environ}
\usepackage{caption}
\usetikzlibrary{arrows}
\usepackage{xifthen}
\usepackage{comment}
\usepackage{marginnote}

\newcommand{\R}{\mathbb{R}}

\newcommand{\N}{\mathbb{N}}

\newcommand{\Borelcodes}[1]{\textbf{BC}^{#1}}
\newcommand{\Borelfunctions}[1]{\textbf{TBF}^{#1}}

\newcommand{\Analytic}[1]{\textbf{A}^{#1}}
\newcommand{\Coanalytic}[1]{\textbf{C}^{#1}}
\newcommand{\Hausdorff}[1]{\mathcal{H}^{#1}}

\newtheorem*{theorem*}{Theorem}
\newtheorem{theorem}{Theorem}[section]
\newtheorem{lemma}[theorem]{Lemma}

\newtheorem{cor}[theorem]{Corollary}
\newtheorem{conj}[theorem]{Conjecture}
\newtheorem{prop}[theorem]{Proposition}

\theoremstyle{definition} \newtheorem{remark}[theorem]{Remark}
\newtheorem*{defin*}{Definition}
\newtheorem{defin}[theorem]{Definition}

\DeclareMathOperator{\proj}{proj}
\DeclareMathOperator{\gr}{graph}
\DeclareMathOperator{\bel}{int}
\DeclareMathOperator{\graph}{graph}

\DeclareMathOperator{\essupp}{ess-\sup}
\DeclareMathOperator{\essup}{ess^{\al}-\sup}
\DeclareMathOperator{\essuprime}{ess^{\al'}-\sup}
\DeclareMathOperator{\essupk}{ess^{\mathit{k}}-\sup}
\DeclareMathOperator{\essupno}{ess^1-\sup}

\DeclareMathOperator{\essupep}{ess^{\al-\ep}-\sup}

\def\leb{\mathcal{L}}
\def\hau{\mathcal{H}}

\def\garo{\mathcal{G}_{\rho}}

\newcommand{\al}{\alpha}
\newcommand{\be}{\beta}
\newcommand{\ga}{\gamma} 
\newcommand{\de}{\delta}
\newcommand{\eps}{\varepsilon} 
\newcommand{\ep}{\varepsilon} 
\newcommand{\om}{\omega}
\newcommand{\si}{\sigma}

\newcommand{\la}{\lambda}

\def\ti{\widetilde}

\def\lkb{\lesssim}

\def\rr{{\mathbb R}}
\def\zz{{\mathbb Z}}

\def\R{{\mathbb R}}

\def\su{\subset}

\def\se{\setminus}

\def\al{\alpha}
\def\be{\beta}
\def\ga{\gamma}
\def\Ga{\Gamma}
\def\de{\delta}

\def\ep{\varepsilon}
\def\si{\sigma}

\def\ol{\overline}
\def\ti{\widetilde}
\def\diam{{\rm diam}\, }
\def\dim{{\rm dim}\, }

\def\leb{\mathcal{L}}
\def\hau{\mathcal{H}}
\def\u{\mathcal{U}}
\def\s{\mathcal{S}}

\def\iq{\mathcal{Q}}
\def\ir{\mathcal{R}}
\def\ic{\mathcal{C}}
\def\ib{\mathcal{B}}
\def\ij{\mathcal{J}}

\def\net{\mathcal{N}}

\def\hal{\hau^{\al}}

\def\gal{\Ga_{\al}}

\def\lkb{\lesssim}

\providecommand{\semmi}[1]{}
\begin{document}

\title[A Fubini-type theorem for Hausdorff dimension]{A Fubini-type theorem for Hausdorff dimension}
\author{K. H\'era, T. Keleti, and A. M\'ath\'e}

\address
{Alfr\'ed R\'enyi Institute of Mathematics, 
Re\'altanoda utca 13-15, H-1053 Budapest, Hungary}

\email{herakornelia@gmail.com}

\address
{Institute of Mathematics, E\"otv\"os Lor\'and University, 
P\'az\-m\'any P\'e\-ter s\'et\'any 1/c, H-1117 Budapest, Hungary}

\email{tamas.keleti@gmail.com}

\address
{Mathematics Institute, University of Warwick, Coventry, CV4 7AL, UK}

\email{a.mathe@warwick.ac.uk}

\thanks{This research was supported  
by the Hungarian National Research, Development and Innovation Office -- NKFIH, 124749.
The second author is grateful to the Alfr\'ed R\'enyi Institute,
where he was a visiting researcher during a part of this project.
He was also partially supported 
by the Hungarian National Research, Development and Innovation Office -- NKFIH, 129335.
} 

\begin{abstract}
It is well known that a classical Fubini theorem for Hausdorff dimension cannot hold; that is, 
the dimension of the intersections of a fixed set with a parallel family of planes do not determine the dimension of the set. 
Here we prove that a Fubini theorem for Hausdorff dimension does hold modulo sets that are small on all Lipschitz graphs. 

We say that $G\subset \R^k\times \R^n$ is $\Gamma_k$-null if for every Lipschitz function $f:\R^k\to \R^n$ the set $\{t\in\R^k\,:\,(t,f(t))\in G\}$ has measure zero. We show that for every Borel set $E\subset \R^k\times \R^n$ 
with $\dim (\proj_{\R^k} E)=k$  there is a $\Gamma_k$-null subset $G\subset E$ such that
$$\dim (E\setminus G) = k+\essupp(\dim E_t)$$
where $\essupp(\dim E_t)$ is the essential supremum of the Hausdorff dimension of the vertical sections $\{E_t\}_{t\in \R^k}$ of $E$. 
In addition, we show that, provided that $E$ is not $\Gamma_k$-null, there is a $\Gamma_k$-null subset $G\subset E$ such that for $F=E \setminus G$, 
the Fubini-property holds, that is, $\dim (F) = k+\essupp(\dim F_t)$. 

We also obtain more general results by replacing $\R^k$ by an Ahlfors--David regular set.
Applications of our results include Fubini-type results for unions of affine subspaces,
connection to the Kakeya conjecture and 
projection theorems. 
\end{abstract}

\maketitle

\tableofcontents

\section{Introduction}
\label{intro}

In this paper we prove Fubini-type results for Hausdorff dimension. 
It is well known, see e.~g.~\cite{Fa}, that for Hausdorff dimension, the classical Fubini theorem does not hold. 
More precisely, it is not true in general that if all the vertical sections of a plane set are $\be$-dimensional then the set is $(\be +1)$-dimensional: 
its Hausdorff dimension can be strictly greater than $\be +1$. 
On the other hand, it is also known that this phenomenon cannot appear in many directions simultaneously. In fact, the (naive) Fubini theorem for Hausdorff dimension does hold in almost every direction (see Proposition \ref{marmat}).
 
In this paper we are interested in Fubini type results that hold in every (fixed) direction --- this is also the setting that has relevance to projection theorems.
We prove that every set can be written as a union of a set for which the (naive) Fubini theorem actually holds and a set that is small on Lipschitz curves that are transversal to that fixed direction.



We use the notion of $\Gamma$-null sets, 
sets intersecting each Lipschitz graph in zero measure. 
We prove that if the plane set $B$ is Borel such that its projection onto the horizontal axis has 
Hausdorff
dimension 1, 
then there exists a $\Gamma$-null set $G \su B$ such that the Hausdorff dimension of $B \se G$ equals 1 plus 
the essential supremum of the Hausdorff dimension of the vertical sections of $B$.  
As a corollary (Corollary~\ref{fucor}), we also show that, provided that $B$ is not $\Gamma$-null, there is a $\Gamma$-null subset $G\subset B$ such that for $F=B \setminus G$, 
the Fubini-property holds, that is, the Hausdorff dimension of $F$ equals 1 plus 
the essential supremum of the Hausdorff dimension of the vertical sections of $F$.
More generally, we prove results for sets $B \su \rr^k \times \rr^n$
for which the projection of $B$ to $\rr^k$ is contained in an Ahlfors--David $\alpha$-regular set $T$, where $\alpha$ is the Hausdorff dimension of the projection of $B$ to $\rr^k$.
(Here $\alpha$-regularity of $T$ means that
the $\hal$-measure of the intersection of $T$ with balls centered at the points of $T$ of radius $r$ is approximately $r^{\al}$ for all $r>0$.) 
It turns out that the Ahlfors--David regularity condition is crucial. 

We also study the special case when the Borel set $B\su\R^m$ is the union of lines that are not orthogonal to the first axis.
We formulate the conjecture (Conjecture~\ref{conjunion}) that 
for these sets the Fubini property holds:
the Hausdorff dimension of $B$ equals $1$ plus
the essential supremum of the Hausdorff dimension of the 
intersections of $B$ with the hyperplanes perpendicular to the
first axis.
Among others we prove that this conjecture, if true, 
would imply a weak form of the Kakeya conjecture, and 
we also prove that our conjecture holds 
when the Hausdorff dimenion of $B$ is less than $2$.
In fact, we study not only unions of lines but also unions of
$k$-dimensional affine subspaces in $\R^m$ for any $1\le k<m$.
Using standard duality arguments, our Fubini type results can be translated to projection theorems: for a natural restricted family of projections we give estimates for the dimension of almost every projection of a fixed set (Corollaries~\ref{proj1} and \ref{proj2}).


The paper is organized as follows: 
In Section~\ref{s:mainresults} 
we give the main set-up of our problem 
and present some classical theorems as well as some
observations  
and we state our main Fubini type results Theorems~\ref{thm1} and \ref{thm2} for Hausdorff dimension. 
In the next two sections we show how these results can
be applied: 
in Section~\ref{appl} to unions of affine subspaces and to get projection theorems and in Section~\ref{s:counter-examples} to prove 
the above mentioned Corollary~\ref{fucor} about the Fubini property of a large subset.
In Section~\ref{s:counter-examples} we also 
show counterexamples for possible stronger statements.
In Section \ref{pfthm2} we prove our main results, Theorems \ref{thm1} and \ref{thm2}, subject to two lemmas, Lemma \ref{meas} and Lemma \ref{lem1}. 
The proof of Lemma \ref{meas} is contained in Section \ref{pfmeas} and the Appendix. 
Lemma \ref{lem1} is proved in Section \ref{pflem1}. 

\section{Preliminaries and statements of our main results}
\label{s:mainresults}

\subsection{Notation and basic definitions}
\label{notate}
The open ball of center $x$ and radius $r$ will be denoted by $B(x,r)$ or $B_{\rho}(x,r)$ if we want to indicate the metric $\rho$. 
The closed ball of center $x$ and radius $r$ will be denoted by $\ol{B}(x,r)$. 
For a set $U \su \rr^n$, $U_{\de}=\cup_{x \in U} B(x,\de)$ denotes the open $\de$-neighborhood of $U$, and $\diam (U)$ denotes the diameter of $U$. 

Let $s \geq  0$, $\de \in (0,\infty]$ and $A \su \rr^n$. By the $s$-dimensional Hausdorff $\de$-premeasure of $A$ we mean  
$$\hau^s_{\de}(A)= \inf \left\{\sum_{i=1}^{\infty} (\diam(U_i))^s : A \su \bigcup_{i=1}^{\infty} U_i, \ \diam(U_i) \leq \de \ (i=1,2,\dots)\right\}.$$ 
The $s$-dimensional Hausdorff measure of $A$ is defined as $\hau^s(A)=\lim_{\de \to 0} \hau^s_{\de}(A)$.
The Hausdorff dimension of $A$ is defined as 
$$\dim A=\sup\{s: \hau^s(A)>0\}.$$ 
For the well known properties of Hausdorff measures and dimension, see e.g. \cite{Fa}. 

For a finite set $A$, let $|A|$ denote its cardinality. 

We will use the notation $a \lesssim_{\alpha} b$ if $a \leq Cb$ where $C$ is a constant depending on $\alpha$. 
If it is clear from the context what $C$ should depend on, we may write only $a \lesssim b$.

For any integers $1 \leq  k < m$, let $G(m,k)$ / $A(m, k)$ denote the space
of all $k$-dimensional linear / affine subspaces of $\rr^m$, respectively. 

We will denote the orthogonal projection from $\rr^m$ onto a subspace $V \in G(m,k)$ by $\proj_V$. 
We will write $\rr^k$ for the subspace $\rr^k \times \{0\}$ in $\rr^m$, and write $\proj_{\rr^k}$ for the orthogonal projection onto it.

\subsection{The Fubini-type setting}
\label{fubinitype}
In this section we define the set-up for the Fubini-type problem for Hausdorff dimension and present some classical theorems as well as some easily derivable observations. 

Fix two integers $n,k \geq 1$, let $B \su \rr^{k} \times \rr^n$, and for any 
$t \in \rr^k$ let 
$$B_t=\{x \in \rr^n: (t,x) \in B\},$$
the ``vertical'' section of $B$ corresponding to $t$. 
The following theorem, which says that if a set $B$ has many large sections then $B$ is large, is well known (see e.g. \cite{Fa}). 

\begin{theorem}
\label{classic}
For any 
$B \su \rr^{k} \times \rr^n$,
$0 \leq \al \leq k$ and $0 \leq \be \leq n$, if there exists 
$A \su \rr^k$ with $\dim A \geq \al$ such that for all $t \in A$, $\dim B_t \geq \be$, then $\dim B \geq \al + \be$. 
\end{theorem}

On the contrary, even if all the sections of $B$ are small, that does not imply that $B$ is small. 
For example, it can happen that all sections of $B$ are singletons, and the dimension of $B$ is equal to 
the dimension of the ambient space. 
This is the context of the next theorem (see e.g. \cite{MW} and \cite{FH}). 

\begin{theorem}
\label{fctn}
There exists a continuous function $f: [0,1] \to [0,1]^n$ whose graph $B=\{(t,f(t)): t \in [0,1]\}$ has Hausdorff dimension $n+1$. 
\end{theorem}

In this paper we will prove a Fubini-type result of the following nature: if all the sections of $B$ have small Hausdorff dimension, 
then we can remove a small set $G \su B$ (the specific notion of smallness we use will be defined later) 
such that the remaining set, $B \se G$ has small Hausdorff dimension. 

\begin{defin}\label{d:ess-sup}
For any non-empty $B \su \rr^k \times \rr^n$, let $A=\proj_{\rr^k}B \su \rr^k$ and  $\al=\dim A$. 
If $\hal(A)>0$, we define 
\begin{align}
\label{eqessup}
\essup(\dim B_t)= & \sup\{q \geq 0: \hal(\{t \in \rr^k: \dim B_t > q \}) >0\} \nonumber \\ 
=& \inf\{q \geq 0: \hal(\{t \in \rr^k: \dim B_t > q \}) =0\},
\end{align}
where $\sup \emptyset$ is taken to be zero.

If $\hal(A)=0$, we let 
\begin{equation}
\label{essnull}
\essup(\dim B_t)=\lim_{\ep \to 0+} \essupep (\dim B_t). 
\end{equation}
\end{defin}

\begin{remark}
\label{true}
Note that Theorem \ref{classic} implies that if $B$ is as 
in Definition~\ref{d:ess-sup}, 
then $\dim B \geq \al + \essup(\dim B_t)$. 
\end{remark}
Intuitively, we can say that ``$B$ behaves nicely along $\rr^k$'', if the reverse inequality holds as well. 
For this we introduce the following definition.

\begin{defin}[Fubini property]
\label{fuprop}
We say that a nonempty set 
$B \su \rr^k \times \rr^n$ has the 
\emph{Fubini property}, if $\dim B = \al + \be$, where $\al= \dim \proj_{\rr^k} B$, and $\be = \essup(\dim B_t)$. 
\end{defin}

We can naturally extend 
Definitions~\ref{d:ess-sup} and \ref{fuprop}
for any fixed $k$-dimensional subspace $V \in G(k+n,k)$ in place of $\rr^k$. Let $W=V^{\perp}$, 
$\emptyset\neq B \su \rr^{k+n}$, 
and 
$\al=\dim \proj_{V} B$.
If $\hal( \proj_{V} B)>0$,
we define 
\begin{align}
\label{essupv}
\essup(\dim B_{V,t})=& \sup\{q  \geq 0:  \hal(\{t \in V: \dim (B \cap (W +t)) > q \}) >0\} \nonumber \\ 
=& \inf\{q \geq 0: \hal(\{t \in V: \dim  (B \cap (W +t)) > q \}) =0\}.
\end{align}
Similarly as in \eqref{essnull}, if $\hal( \proj_{V} B)=0$, we define 
\begin{equation}
\label{essnullv}
\essup(\dim B_{V,t})=\liminf_{\ep \to 0} \essupep (\dim B_{V,t}). 
\end{equation}
Finally, we say that $B$ has the 
\emph{Fubini propery along} $V$ if
\begin{equation}
\label{aav}
\dim B = \al +  \essup(\dim B_{V,t}). 
\end{equation}

Now we prove that any Borel set 
``behaves nicely'' along almost every subpace, 
using an easy application of the projection theorem 
and the instersection theorem 
of Marstrand and Mattila. 

\begin{prop}
\label{marmat}
Any nonempty Borel set $B\su\R^{k+n}$ has the
Fubini property along almost every $k$-dimensional
subspace; that is,
\eqref{aav} holds for 
$\ga_{k+n,k}$-almost every $V \in G(k+n,k)$, 
where $\al= \dim \proj_{V} B$ and 
 $\ga_{k+n,k}$ denotes the natural measure on the Grassmannian. 
\end{prop}

\begin{proof}
Let $\dim B =\be $.

Assume first that $ \be \leq k$. 
Then by the projection theorem of Marstrand and Mattila, we obtain that 
for $\ga_{k+n,k}$-almost all $V \in G(k+n,k)$, $\dim \proj_V B=\be$. 
Clearly, we have $\essup(\dim B_{V,t})=0$, otherwise by Remark \ref{true} we would get $\dim B > \be$. Thus 
$\dim B = \be + 0$, which is precisely \eqref{aav} in this case. 

Assume now that $\be >k$. Let $\ep>0$ such that $\be - \ep >k$. 
By the projection theorem of Marstrand and Mattila, we obtain that 
for $\ga_{k+n,k}$-almost all $V \in G(k+n,k)$, $\hau^k(\proj_V B)>0$. 
Moreover, we have $\hau^{\be-\ep}(B)>0$, thus by the 
intersection theorem of Marstrand and Mattila (see e.g. \cite[Theorem 10.10]{Ma}), we have 
 for $\ga_{k+n,k}$-almost all $V \in G(k+n,k)$ that 
 $$\hau^k(\{t \in V: \dim (B \cap (W +t)) \geq \be-\ep-k \})>0,$$ 
where $W=V^{\perp}$. This implies by the definition in \eqref{essupv} that 
$\essupk(\dim B_{V,t}) \geq \be-\ep-k$ for each $\ep>0$, thus 
$\essupk(\dim B_{V,t}) \geq \be-k$. 
Combining the above two facts, we obtain that for $\ga_{k+n,k}$-almost all $V \in G(k+n,k)$, 
$$\dim B = \be = k + (\be-k) \leq  k +\essupk(\dim B_{V,t}).$$ 
In fact, by Theorem \ref{classic}, this is an equality, and \eqref{aav} follows. 

\end{proof}

Our aim in this paper is to prove a Fubini-type theorem, which holds along any fixed subspace $V$. 
For simplicity, we only state our results in the context where $V= \rr^k \times \{0\}$.

\subsection{Our main theorems}
\label{main}

We start by defining the notion of $\Ga$-nullness, a notion which has crucial role in our results. 

\begin{defin}[$\Ga_{\al}$-null sets]
Let $0 < \al \leq k$. We say that $G \su \rr^k \times \rr^n$ is \emph{$\Ga_{\al}$-null} if for every Lipschitz map $f: \rr^k \to \rr^n$,
\begin{equation}
\hau^{\al}(\{t \in \rr^k: (t,f(t)) \in G\})=0.
\label{gamma}
\end{equation}

\end{defin}

\begin{remark}
\label{gacup}
For each $\alpha$, the $\Gamma_\alpha$-null sets form a $\sigma$-ideal. In particular, if $G_i$ is $\gal$-null ($i=1,2,\dots$), then $G=\cup_{i=1}^{\infty} G_i$ is also $\gal$-null. 
Moreover, if $\al_1 < \al$ and $G$ is $\Ga_{\al_1}$-null, then $G$ is $\gal$-null as well. 
\end{remark}

\begin{remark}
\label{gahist}
Note that a similar but different notion of $\Ga$-nullness appears in the works of 
Alberti, Cs\"ornyei, Maleva, Preiss, Ti\v{s}er, Zají\v{c}ek, and others
related to the generalization of Rademacher's theorem and differentiability of Lipschitz maps, see e.g. 
\cite{PZ}, \cite{CDT}, \cite{ACP}, \cite{Pr}, \cite{MP}. 
In our paper, it is important that we have a fixed direction (``vertical'') and we consider Lipschitz curves that are transversal to that direction, in particular, we consider all Lipschitz graphs (using the standard coordinate system). In the above mentioned works, on the other hand, there is usually no constraint on the directions of the curves, and often only a subfamily $\Ga$ of Lipschitz curves is used. 
\end{remark}

In our results, the notion of Ahlfors--David regular sets
(see e.~g.~\cite[p.~92]{Ma})
will also be very important. 

\begin{defin}[Ahlfors--David regular sets]
\label{defdefahl}
For any $0 < \al \leq k$, we say that  the Borel set $T \su \rr^k$ is \emph{Ahlfors--David $\al$-regular}, if 
there exist $c, c' >0$ such that 
\begin{equation}
c \cdot r^{\al} \leq \hau^{\al}(T \cap B(t,r)) \leq c'\cdot r^{\al}
\label{ahldef}
\end{equation}
for all $t \in T$ and $0<r \leq 1$.
\end{defin}

\begin{remark}
\label{ahlregm}
Clearly,
if $T \su \rr^k$  is Ahlfors--David $\al$-regular, then 
$\hal(T)>0$, and $\hal|_T$ is $\si$-finite. 
\end{remark}

Our main result is the following. 
\begin{theorem}
\label{thm1}
Let $0<\alpha\le k$ and let $T\subset \mathbb{R}^k$ be an Ahlfors--David $\alpha$-regular set. Let $B\subset T\times \R^n$ be Borel such that $\dim \proj_{\rr^k} B = \al$.
Then
there exists a  $\Ga_{\al}$-null  Borel set $G \su B$ such that 
$$\dim (B \setminus G)= \al + \essup(\dim B_t).$$
(Recall Definition~\ref{d:ess-sup} for the notation $\essup(\dim B_t)$.)
\end{theorem}


Taking $\alpha=k$ and $T=\R^k$, Theorem~\ref{thm1} implies the following.

\begin{cor}
\label{corr}
Let $B \su \rr^k \times \rr^n$ be Borel such that $\dim(\proj_{\rr^k} B)=k$. 
Then there exists a Borel set $G \su B$ such that the following hold: 
\begin{itemize}
\item $G$ is $\Ga$-null, that is, for any $f: \rr^k \to \rr^n$ Lipschitz map, 
$\leb^{k}(t \in \rr^k: (t,f(t)) \in G)=0$.
\item $\dim (B \setminus G)= k + \essupp(\dim B_t)$, where 
$$\essupp(\dim B_t)= \sup\{q \geq 0: \leb^k(\{t \in \rr^k: \dim B_t > q \}) >0\}.$$ 
\end{itemize}
\end{cor}

In order to prove Theorem \ref{thm1} we will use Theorem \ref{classic} and  the following.

\begin{theorem}
\label{thm2}
Let $0<\alpha\le k$ and let $T\subset \mathbb{R}^k$ be an Ahlfors--David $\alpha$-regular set. Let $B\subset T\times \R^n$ be a Borel set
and assume that $0 < \be \leq n$ is a real number such that $\dim B_t \leq \be$ for all $t \in T$. 
Then there exists a $\Ga_{\al}$-null Borel set $G \su B$ such that $\dim (B \setminus G) \leq \al + \be$.
\end{theorem}

In Section \ref{pfthm2} we prove Theorem~\ref{thm2} subject to two lemmas and we show how the combination of Theorem \ref{classic} and Theorem \ref{thm2} implies Theorem \ref{thm1}.



In Section~\ref{s:counter-examples} we show that the Ahlfors--David regularity assumption cannot be dropped in Theorems \ref{thm1} and \ref{thm2}.

\section{Applications to unions of affine subspaces and projection theorems}
\label{appl}

\subsection{Results for unions of affine subspaces}
\label{unions}
In this section we present Fubini-type results for the Hausdorff dimension of unions of affine subspaces as applications of Theorem \ref{thm1}. 
First we introduce some notations. 

Let $2 \leq m$, $1 \leq k <m$ be integers. Recall that $A(m,k)$ denotes the space of all $k$-dimensional affine subspaces of $\rr^m$. 

\begin{defin}
\label{nonvert}
Let $V \in A(m,m-k)$. We say that $P \in A(m,k)$ is 
\emph{non-vertical with respect to $V$}
if the orthogonal projection of $P$ to $V^{\perp}$ is equal to $V^{\perp}$. 

In the case $V=\{0\}\times \rr^{m-k}$ we say simply that $P \in A(m,k)$ is \emph{non-vertical}. 
\end{defin}

The following is an immediate corollary of Theorem \ref{thm1}.
\begin{cor}
\label{corunion}
Fix a non-empty  collection $E \su A(m,k)$ 
of  non-vertical affine subspaces such that 
$B=\bigcup_{P \in E} P \su \rr^m$ is 
Borel, 
and let $\be=\essupk(\dim B_t)$.  

Then for each $P \in E$, there exists $P' \su P$ with $\hau^k(P \setminus P')=0$ such that for 
 $B'=\bigcup_{P \in E} P' \su B$, we have $\dim B' = \be+k$.  
\end{cor}

\begin{proof}
We apply Corollary \ref{corr} for $k$ and $n=m-k$, $B=\bigcup_{P \in E} P \su \rr^m$. 
Clearly, $\proj_{\rr^k} B = \rr^k$. 
We obtain a $\Ga$-null set $G \su B$ such that 
$\dim (B \setminus G)= k + \be$. 
Since $G$ is $\Ga$-null, we have $\hau^k(G \cap P)=0$ for each $P \in E$, thus taking $P'=P \se G$ completes the proof. 
\end{proof}

For simplicity, we introduce the following notation. 

\begin{defin}
\label{defir}
For $E \su A(m,k)$ let 
\begin{align}
r(E)=\inf\{\dim B'-k: \ B' \su \rr^m \ \text{such that} \ \hau^k(B' \cap P) > 0 \ \text{for each} \ P \in E \}. 
\label{eqir}
\end{align}
\end{defin}

We prove the following Fubini-type result about the Hausdorff dimension of unions of affine subspaces, 
which is a consequence of Theorem \ref{thm2}. 

\begin{cor}
\label{corf}
Fix a non-empty collection $E \su A(m,k)$ 
of non-vertical affine subspaces such that 
$B=\bigcup_{P \in E} P \su \rr^m$ is Borel. 
Then for $\leb^k$-almost all $t \in \rr^k$, $\dim B_t \geq r(E)$, where $r(E)$ is from \eqref{eqir}. 
\end{cor}

\begin{proof}
Fix $E$ with the above properties and let $r=r(E)$. 
Assume on contrary that $\leb^k(t \in \rr^k: \dim B_t < r) >0$. Then there exists $\ep>0$ such that 
$\leb^k(t \in \rr^k: \dim B_t < r-\ep) >0$.
Note that these sets are measurable: they are in the $\sigma$-algebra generated by analytic sets (see \cite{De} or \cite[Remark 6.2.]{MM}). 
Therefore we can choose a compact set $A' \su \{t\in \rr^k: \dim B_t < r-\ep\}$ with $\leb^k(A')>0$. Let 
$B'=B \cap (A' \times \rr^{m-k})$. 
Clearly, $\hau^k(B' \cap P) > 0$ for each $P \in E$, and $B'$ is Borel. 
We apply Theorem \ref{thm2} for $B'$ and  obtain a $\Ga_k$-null set $G' \su B'$ such that 
$\dim (B' \se G') \leq k + r - \ep$. Then $\hau^k(G' \cap P) = 0$ for each $P \in E$, 
thus $\hau^k((B'\se G') \cap P) > 0$ for each $P \in E$. By the definition of $r=r(E)$ (see \eqref{eqir}), this implies that 
$$ k + r \leq \dim (B' \se G') \leq k + r - \ep,$$
which is a contradiction and we are done. 
\end{proof}

Now we list some results aiming to find lower bounds for $r(E)$  (see \eqref{eqir}) for $E \su A(m,k)$. As these results will involve the Hausdorff dimension of the collection $E$, we should define a metric on $A(m,k)$.

For $P_i=V_i + a_i \in A(m,k)$, where $V_i \in G(m,k)$ and $a_i \in V_i^{\perp}$, $i=1,2$, we put 
$$\de(P_1,P_2)=\|\proj_{V_1}-\proj_{V_2}\| + |a_1-a_2|$$ 
where $\| \cdot \|$ denotes the standard operator norm. 
Then $\de$ is a metric on $A(m,k)$, see \cite[p. 53]{Ma}.  

\begin{defin}
Let $\rho$ be a metric on $A(m,k)$. We say that $\rho$ is a \emph{natural metric} if 
$\rho$ and $\de$ are Lipschitz equivalent, that is, if  
there exist positive constants $K_1$ and $K_2$ such that, for every $P_1,P_2 \in A(m,k)$,
$K_1 \cdot \de(P_1,P_2)\leq \rho (P_1,P_2) \leq K_2 \cdot \de(P_1,P_2).$
\end{defin}

The following theorem was proved by Falconer and Mattila in \cite{FM} for $k=m-1$, and by the authors in \cite{HKM} for arbitrary $0 < k \leq m-1$. 

\begin{theorem}[Falconer--Mattila, H\'era--Keleti--M\'ath\'e]
\label{hkmthm}
Let $1 \leq k <m$ be integers, fix a natural metric on $A(m,k)$, and let  $E \su A(m,k)$ be nonempty. 
If $B \su \rr^m$ satisfies $\hau^k(B \cap P) > 0$ for each $P \in E$, then 
$$ \dim  B \geq k+\min\{\dim E,1\}.$$

Moreover, if $\dim E  \in [0,1]$ and $B \su \bigcup_{P \in E} P$ with $\hau^k(B \cap P) > 0$ for each $P \in E$, then 
$$\dim B = k + \dim E.$$
\end{theorem}

This means that if $E \su A(m,k)$ with $\dim E \leq 1$, then  $r(E)=\dim E$ (see \eqref{eqir}).

\begin{remark}
Note that the inequality $\dim \left(\cup_{P \in E} P \right) \leq k +\dim E$ holds for every collection $E$ (see e.g. \cite{HKM}). However, equality can fail when $\dim E>1$.
%
\end{remark}

In \cite{H}, the first author gave another estimate for the dimension of sets of the above type. 

\begin{theorem}[H\'era]
\label{hthm}
Let $1 \leq k <m$ be integers, fix a natural metric on $A(m,k)$, and let $E \su A(m,k)$ be nonempty. 
If $B \su \rr^m$ such that $\hau^k(B \cap P) > 0$ for each $P \in E$, then 
$$\dim  B \geq k+\frac{\dim E}{k+1}.$$

In other words, for any $E \su A(m,k)$ we have $r(E) \geq \frac{\dim E}{k+1}$ (see \eqref{eqir}). 
\end{theorem}

\begin{remark}
It is not hard to check (see \cite{H}) that there are sets $E \su A(m,k)$ such that $r(E)= \frac{\dim E}{k+1}$.
\end{remark}

Now we prove that 
unions $B=\cup_{P \in E} P$ of small enough dimension
have the Fubini property.

\begin{cor}
\label{corsmall}
Fix a non-empty collection $E \su A(m,k)$ 
of  non-vertical affine subspaces such that $\dim E \leq 1$, and  $B=\bigcup_{P \in E} P \su \rr^m$ is Borel. 
Then $B$ has the Fubini property; that is,
$\dim B = \be+k$, where $\be=\essupk(\dim B_t)$.

In particular, if $B$ is a Borel union of non-vertical $k$-dimensional affine subspaces with $\dim B < k+1$, then 
it has the Fubini property.
\end{cor}

\begin{proof}
Since $\dim E \leq 1$,  we have $r(E)=\dim E$ by Theorem \ref{hkmthm}. 
Applying Corollary \ref{corf}, we obtain that 
for $\leb^k$-almost all $t \in \rr^k$, $\dim B_t \geq \dim E$. In particular, $\be=\essupk(\dim B_t) \geq \dim E$. 
Combining this with Remark \ref{true} and Theorem \ref{hkmthm}, we get  
$\be +k \leq \dim B = \dim E + k \leq \be +k$, so $B$ has the Fubini-property.  

Note that $\dim B < k+1$ implies $\dim E < 1$ by Theorem \ref{hkmthm}, thus the above argument can be applied. 
\end{proof}

We also formulate the conjecture that for arbitrary Borel unions of $k$-dimensional affine subspaces
Fubini property holds:

\begin{conj}
\label{conjunion}
Fix a non-empty collection $E \su A(m,k)$ 
of  non-vertical affine subspaces such that $B=\bigcup_{P \in E} P \su \rr^m$ is Borel. 
Then 
\begin{equation}
\label{kakeq}
\dim B = \be+k, \ \text{where} \ \be=\essupk(\dim B_t).
\end{equation}
\end{conj}

Recall that a subset of $\R^m$ is called a \emph{Besicovitch set} if it contains a unit line segment in every direction,
and that the \emph{Kakeya conjecture} states that such a set must have
Hausdorff / upper Minkowski dimension $m$. 
The following theorem connects Conjecture~\ref{conjunion} 
to the Kakeya conjecture. 

\begin{theorem}\label{t:kakeya}
(i) For any $m\ge 2$, if Conjecture~\ref{conjunion} holds for that $m$ and
$k=1$, then every Besicovitch set in $\R^m$ has
Hausdorff dimension at least $m-1$.

(ii) If Conjecture~\ref{conjunion} holds for $k=1$ for every 
$m\ge 2$ then any Besicovith set in $\R^m$ has upper Minkowski dimension $m$, so Kakeya conjecture holds for Minkowski dimension. 
\end{theorem}

%

\begin{proof}
A well-known argument (see e.g.~\cite{Ke}) gives that
if for every $m\ge 2$ every Besicovitch set in $\R^m$ has Hausdorff dimension at least $m-1$ then every 
Besicovitch set in $\R^m$ has upper Minkowski dimension $m$.
Therefore it is enough to prove (i). 

Suppose now that there exists a Besicovitch set in $\R^m$ of Hausdorff dimension $d<m-1$.

We will rely on a recent result of the second and third authors \cite{KM} that states
that if there is a Besicovitch set of dimension $d$, then there is a closed set $A\su\R^m$ of 
Hausdorff dimension $d$ such that $A$ contains a line
in every direction.
Crucially, in this proof, this allows us to replace line segments by full lines. The additional property that $A$ is now a closed set helps us evade complicated measurability arguments.

A vague sketch for the rest of the proof is the following. With the help of Theorem~\ref{fctn} we choose a very special set $D$ of directions (of dimension $m-1$). We pass to a subset $A'$ of $A$ corresponding to lines going in this set of directions. We extend these lines in the projective space (by adding their direction to the hyperplane at infinity) to obtain $A'\cup D$. We apply a projective transformation $\Phi$ that maps the hyperplane at infinity to a particular standard hyperplane of $\R^m$. The projective image $\Phi(A'\cup D)$ is a union of lines, it still has dimension $m-1$; however, every plane in a suitable family of parallel planes will intersect $\Phi(D)$ in at most one point only, making $\Phi(D)$ essentially ``invisible''. As $\Phi(A')$ has dimension at most $d< m-1$, this will contradict Conjecture~\ref{conjunion} for $\Phi(A'\cup D)$.

For the precise proof, let $H_1$ denote the hyperplane $\{0\} \times \rr^{m-1}$, and let $w=(1,0,\dots,0) \in \rr^m$. 
For any $p \in H_1$ and $v \in S^{m-1}$, let $\ell(p,v)$ denote the line though $p$ in the direction of $v$. 
Using a projective map that takes the hyperplane at infinity to $H_1$, a well known argument (see e.g. \cite{Ke}) gives that there is $A' \su \rr^m$ with $\dim A'=d$ such that for each $p \in H_1$, 
there exists a line $\ell(p,v)$ through $p$ such that $v \not\perp w$ and $\ell(p,v) \se \{p\} \su A'$. 
Let $F=A' \cup H_1$, it is easy to see that since $A$ is closed, $F$ is closed.  
Let $f: [0,1] \to \rr^{m-2}$ be a continuous function such that $\dim graph(f) = m-1$ (see Theorem~\ref{fctn}), and let 
\begin{equation}
\label{eq:cdef}
C=\{(0,t,f(t)): t \in [0,1]\} \su  H_1.
\end{equation}
Clearly, $C$ is compact, and $\dim C = m-1$. 

We will use the following lemma. 
\begin{lemma}
\label{closed}
Let $F \su \rr^m$ closed, $C \su F$ compact, $u, w \in S^{m-1}$, and $\de>0$. Then 
$$K=\bigcup\{\ell(p,v): p \in C, \ v \in S^{m-1}, \ \ell(p,v) \su F, \ |\langle v,u \rangle| \geq \de, \ |\langle v,w \rangle| \geq \de\}$$ is closed. 
\end{lemma}

\begin{proof}
Let $a_n \in K$, $a_n \to a$. Then there exists a line $\ell(p_n,v_n) \su F$ such that $a_n \in \ell(p_n,v_n)$, $p_n \in C, v_n \in S^{m-1}, |\langle v_n,u \rangle| \geq \de, |\langle v_n,w \rangle| \geq \de$. Since $C$ and $S^{m-1}$ are compact, we can assume that there is $p \in C$, $v \in S^{m-1}$ such that $p_n \to p, v_n \to v$. Then $a \in \ell(p,v)$, $|\langle v,u \rangle| \geq \de$, $|\langle v,w \rangle| \geq \de$, and since $F$ is closed, $\ell(p,v) \su F$, so $a \in K$.
\end{proof}

Let $H_2$ denote the $(m-2)$-plane $\{(0,0)\} \times \rr^{m-2}$. 
Pick $u_1, u_2 \in S^{m-1}$ such that $u_1$ and $u_2$ are orthogonal to $H_2$, and $u_1$, $u_2$, and $w$ are pairwise non-parallel. 
For $i=1,2$ and $k \in \N$, let 
$$B_{i,k}=\bigcup\{\ell(p,v): p \in C, \ v \in S^{m-1}, \ \ell(p,v) \su F, \ |\langle v,u_i \rangle| \geq 1/k, \ |\langle v,w \rangle| \geq 1/k\},$$
where $F=A' \cup H_1$, $w=(1,0,\dots,0)$, and $C$ is constructed in \eqref{eq:cdef}. 
By Lemma \ref{closed}, $B_{i,k}$ is closed for each $i$ and $k$. One easily checks that by construction, for each $p \in C$, there is a line 
$\ell(p,v)$ through $p$ such that $\ell(p,v) \su F$,  $|\langle v,w \rangle|>0$, and $|\langle v,u_i \rangle| >0$ for $i=1$ or $i=2$, thus $C \su \bigcup_{i=1,2} \bigcup_{k \in \N} B_{i,k}$. 
Since $\dim C = m-1$, there are $i$ and $k$ such that $\dim B_{i,k} \geq m-1-\varepsilon$, for an arbitrary $\varepsilon>0$. Fix such an $i$ and $k$ and let $B=B_{i,k}$, $B'=B \setminus H_1$. Since $B' \su A'$, $\dim B' \leq d$. Let $u=u_i$, let $U$ denote the line generated by $u_i$ and let $V=U^{\perp}$. 
By construction, $\proj_{U} \ell(p,v) =U$ for any line $\ell(p,v)$ in the definition of $B$. Moreover, $V \neq H_1$ is a hyperplane containing $H_2$, thus $C \cap (V+t \cdot u)$ is at most one point for each $t \in \rr$. This, combined with $B \setminus B' \su C$ gives that 
\begin{equation}
\dim ( B \cap  (V+t \cdot u)) = \dim (B' \cap (V+t \cdot u)) \ \text{for each} \ t \in \rr. 
\label{eqkak}
\end{equation}
By the definition in \eqref{essupv}, \eqref{eqkak} implies that 
$\essupno(\dim B_{U,t})=\essupno(\dim B'_{U,t})$. Call this value $\be$. Then, using Conjecture \ref{conjunion} for $B$ and Remark \ref{true} for $B'$ we have 
$$ m-1 -\varepsilon\leq \dim B = \be +1 \leq \dim B' \leq d,$$ which, since $\varepsilon>0$ was arbitrary, completes the proof.  
\end{proof}

\begin{remark}
As one can see in the above proof, it is in fact enough to assume  in the statement of Theorem \ref{t:kakeya} that Conjecture \ref{conjunion} holds for closed unions $B$ of non-vertical $k$-dimensional affine subspaces.
\end{remark}

\subsection{Projection theorems as corollaries}
\label{projc}

Fix two integers $n,k \geq 1$, and fix $t=(t_1,\dots,t_k) \in  \rr^k$. 
For any $x \in \rr^{(k+1)n}$, we use the notation $x=(x_0,x_1,\dots,x_k)$, where $x_i \in \rr^n$ for $i=0,\dots,k$. Let 

\begin{equation}
\pi_t: \rr^{(k+1)n} \to \rr^n, (x_0,x_1,\dots,x_k) \mapsto x_0 + t_1 \cdot x_1 + \dots + t_k \cdot x_k.
\label{eqpi}
\end{equation}

The projections $\{\pi_t\}_{t \in \rr^k}$ were first investigated by Oberlin, see \cite{Ob2}. His results imply the following theorem, see also \cite{Ma17}. 

\begin{theorem*}[Oberlin]
Let $X \su \rr^{(k+1)n}$ be non-empty Borel. 
\begin{itemize}
\item
If $\dim X \leq (k+1)n-k$, then for $\leb^k$-almost all $t \in \rr^k$, $\dim \pi_t(X) \geq \dim X - k(n-1)$.
\item
If $\dim X > (k+1)n-k$, then for $\leb^k$-almost all $t \in \rr^k$, $\leb^n(\pi_t(X)) >0$. 
\end{itemize}
\end{theorem*}

Here we prove two results of similar nature. With the help of Corollary~\ref{corf}, we can translate Theorem~\ref{hthm} and Theorem~\ref{hkmthm} to the language of projection theorems.


\begin{cor}
\label{proj1}
Let $X \su \rr^{(k+1)n}$ be non-empty Borel 
with $\dim X \leq 1$. Then for $\leb^k$-almost all $t \in \rr^k$, 
$$\dim \pi_t(X)=\dim X.$$
\end{cor}


\begin{cor}
\label{proj2}
Let $X \su \rr^{(k+1)n}$ be non-empty Borel. 
Then for $\leb^k$-almost all $t \in \rr^k$, 
$$\dim \pi_t(X) \geq \frac{\dim X}{k+1}.$$
\end{cor}

Note that these two corollaries could have been proved by repeating/modifying the original proofs of Theorems~\ref{hthm} and \ref{hkmthm}. The power of our method here is that these corollaries can be obtained directly from the end results, as a ``black box'', without the need to revisit the original proofs.

\begin{proof}[The proof of Corollary \ref{proj1} and \ref{proj2}]
It is enough to prove the corollaries for compact sets $X$, since every Borel set $X$ contains a compact set with dimension arbitrarily close to the dimension of $X$.

From now on we assume that $X$ is compact.

We use the standard duality connection between sections of unions of affine subspaces and 
projections of the ``code set'' associated to collections of affine subspaces.
For $x=(x_0,x_1,\dots,x_k) \in X \su \rr^{(k+1)n}$, let 
$$P_x=\{(t,\pi_t(x)): t \in \rr^k\}.$$ 
Then $P_x$ is a  non-vertical $k$-dimensional affine subspace in $\rr^{k+n}$ for each $x \in X$. 
Let 
$$E=\{P_x: x \in X\} \su A(k+n,k), \ \text{and} \ B=\cup_{P \in E} P \su \rr^{k+n}.$$ 
We claim that $\dim X=\dim E$. Indeed, the Euclidean metric on $X$ defines a metric on $E$. Since $X$ is bounded, this metric is Lipschitz equivalent to the restriction of a natural metric of $A(k+n,k)$ to $E$. So $\dim X=\dim E$ follows.

The set $B$ is a continuous image of $X\times \R^k$. This product set is a union of countably many compact sets, so $B$ is a union of countably many compact sets,  hence $B$ is Borel. (A more careful argument shows that $B$ is closed.)

By construction, we have
\begin{equation}
 B_t=\pi_t(X)\, \text{ for every }\,  t \in \rr^k.
\label{eqdual}
\end{equation}

If $s=\dim X =\dim E \leq 1$ then the combination of Theorem \ref{hkmthm} and Corollary \ref{corf} implies 
that for $\leb^k$-almost all $t \in \rr^k$, $\dim B_t \geq s$. Then by \eqref{eqdual} we obtain 
that 
for $\leb^k$-almost all $t \in \rr^k$, $\dim \pi_t(X)=\dim X$, which completes the proof of Corollary \ref{proj1}. 

To finish the proof of Corollary \ref{proj2}, we use the combination of Theorem \ref{hthm} and Corollary \ref{corf}. 
We obtain that for $\leb^k$-almost all $t \in \rr^k$, $\dim B_t \geq \frac{s}{k+1}$, where $s=\dim X =\dim E$. By \eqref{eqdual} 
this implies that 
for $\leb^k$-almost all $t \in \rr^k$, $\dim \pi_t(X) \geq \frac{\dim X}{k+1}$, 
which completes the proof of Corollary \ref{proj2}. 

\end{proof}

\begin{remark}
Corollary~\ref{proj2} is obviously sharp for the endpoint $\dim X=(k+1)n$, taking $X=\R^{(k+1)n}$ or in fact any other set with this dimension.
\end{remark}

\begin{remark}
There is a rich literature of results concerning certain restricted families of projections, see e.g.~\cite{JJK}, \cite{FO14}, \cite{Or15}, \cite{OO}. 
We remark that our first projection theorem, Corollary \ref{proj1}, can also be proved using a classical potential-theoretic method, similarly as in  \cite[Proposition 1.2.]{FO13}.
\end{remark}

\section{Fubini property and counterexamples}
\label{s:counter-examples}

Recall that we  
say that $B \su \rr^k \times \rr^n$ has the Fubini property if $\dim B = \al + \be$, where $\al= \dim \proj_{\rr^k} B$, and $\be = \essup(\dim B_t)$. 
In this section first we show that any
non-$\Ga_{\al}$-null  
Borel set has a large subset 
for which the Fubini property holds.
More precisely we prove 
the following corollary of Theorem~\ref{thm1}.

%
%

\begin{cor}
\label{fucor}
Let $0 < \al \leq k$, and let $B \su  \rr^k \times \rr^n$
non-empty Borel set such that $\proj_{\rr^k} B \su T$, 
where $T \su \rr^k$ is an Ahlfors--David $\al$-regular set.   
Then there exists a Borel set $G \su B$ such that $G$ is $\Ga_{\al}$-null, and 
$B'=B \se G$ has the Fubini property. 

Moreover, if $B$ is not $\Ga_{\al}$-null, then $\dim \proj B'=\al$, thus in fact, 
$\dim B'= \al + \be'$, where $\be' = \essup(\dim B'_t)$. 
\end{cor}

\begin{proof} 
%
If $B$ is $\Ga_{\al}$-null, then Corollary \ref{fucor} clearly holds: 
we can simply take $B'=\{b\}$ to be any singleton contained in $B$. 

Assume now that $B$ is   
not $\Ga_{\al}$-null, thus $\hau^{\al} (\proj_{\rr^k} B) >0$. 
Put $B_0=B$, and $\be_0=\essup(\dim B_t)$. 
We complete the proof using transfinite induction. The idea is the following. 
We apply Theorem \ref{thm1} for $B$ and obtain a $\Gamma_\alpha$-null Borel set $G_0 \su B$ such that  
$\dim (B \setminus G_0)= \al + \be_0$.  
Put $B_1=B \setminus G_0$. Clearly, $B_1$ is Borel and $B_1 \su B \su T \times \rr^n$. 
Since $B$ is not $\gal$-null, we have that $B_1$ is not $\gal$-null, and clearly, 
this implies that $\hal(\proj_{\rr^k} B_1)>0$, thus $\dim \proj_{\rr^k} B_1 =\al$. 
If $\dim B = \dim B_1$, then $B$ has the Fubini property, we are done. 
If $\dim B_1 < \dim B$, we apply Theorem \ref{thm1} for $B_1$. 
We continue this process for $n=0,1,\dots$, and put $B_{\om_0}=\bigcap_{n=0}^{\infty} B_n$. 
If $B_{\om_0}$ has the Fubini property, then we will be done (see the details below in general), 
otherwise, we apply Theorem \ref{thm1} for $B_{\om_0}$ and continue the process. 

Whenever $B_{\la}$ is an already defined Borel non-$\Ga_{\al}$-null set for an ordinal number $\la$, 
let $\be_{\la}=\essup(\dim (B_{\la})_t)$. 
Let $\ga$ be a successor ordinal, say $\ga=\la + 1$.   
Applying Theorem \ref{thm1} for $B_{\la}$ we obtain a  $\Ga_{\al}$-null Borel set $G_{\la}$ such that 
$\dim (B_{\la} \se G_{\la})= \al + \be_{\la}$. 
Then we let $B_{\ga}=B_{\la} \se G_{\la}$. 
Since $B_{\la}$ is not $\Ga_{\al}$-null, we have that $B_{\ga}$ is not $\Ga_{\al}$-null, and thus $\hau^{\al} (\proj_{\rr^k} B_{\ga})>0$. 
If $\ga$ is a countable limit ordinal, let  $B_{\ga}=\bigcap_{\la < \ga} B_{\la}$. 
Then $B_{\ga}=B \se (\bigcup_{\la < \ga} G_{\la})$. 
Since $\ga$ is a countable ordinal, and $B$ is not $\Ga_{\al}$-null, we have by Remark \ref{gacup} that $B_{\ga}$ is not $\Ga_{\al}$-null, and 
thus $\hau^{\al} (\proj_{\rr^k} B_{\ga})>0$.  
It is also clear that for any $\ga$, the construction yields a Borel set 
$B_{\ga}$ with $B_{\ga} \su B \su T \times \rr^n$, so Theorem \ref{thm1} is applicable for $B_{\ga}$. 

We claim that there exists $\ga < \omega_1$ such that $\dim B_{\ga + 1}=\dim B_{\ga}$, where $\om_1$ denotes the smallest uncountable ordinal. 
Assume on contrary that $\dim B_{\ga + 1} <\dim B_{\ga}$ for all $\ga < \omega_1$. 
Since $B_{\la} \su B_{\ga}$ whenever $\ga \leq \la$, this implies that 
$\dim B_{\la} \leq \dim B_{\ga+1} < \dim B_{\ga}$ for any $\ga < \la$, thus 
$\dim B_{\la}: \om_1 \to \rr^{+}$ is strictly decreasing, which is impossible. 

Let $\ga$ be a countable ordinal such that $\dim B_{\ga + 1}=\dim B_{\ga}$. Then by definition, since $\ga+1$ is a successor ordinal, we have 
$\dim B_{\ga+1} =\al + \be_{\ga}$ and then 
$\dim B_{\ga}= \al + \be_{\ga}$, thus $B_{\ga}$ has the Fubini property. 
Moreover, it is easy to check that 
$B_{\ga}=B \se (\bigcup_{\la < \ga} G_{\la})$, and by Remark \ref{gacup}, 
$\bigcup_{\la < \ga} G_{\la}$ is $\gal$-null, which completes the proof. 
\end{proof}

Note that the second part of Corollary \ref{fucor} 
allows $\be'< \be$ 
and does not claim the uniqueness of $\be'$. 
The next proposition shows that 
we cannot guarantee the existence of a subset with the Fubini property such that $\be'= \be$,
and also that in the second part of Corollary~\ref{fucor} the assumption that $B$ is not $\gal$-null
cannot be dropped. 

\begin{prop}\label{p:Brownian}
(i) There exists a compact set $C\su\R\times\R^2$
with $\proj_{\R} C=[0,1]$
such that for every subset $C'\su C$ with the Fubini
property, $\dim \proj_{\R}C' =0$.

(ii) There exists a non-$\Ga_{1}$-null compact set 
$B \su \rr \times \rr^3$ 
with $\al=\dim(\proj_{\R} B)=1$ and
$\be = \essupno(\dim B_t)=1$,
such that 
for any subset $B' \su B$ possessing the Fubini property, $\al'=0$ or $\be'=0$, where
$\al'=\dim(\proj_{\R} B')$ and
$\be' = \essuprime(\dim B'_t)$.
\end{prop}

\begin{proof}
Let $g:[0,1]\to\R^2$ be a Brownian motion.
By Kaufman's dimension doubling theorem 
(see e.~g.~\cite{MoPe}), almost surely, for any 
set $X\su[0,1]$, we have $\dim g(X) = 2\dim X$.
Fix $g:[0,1]\to\R^2$ with this property.

Then $C=\graph g$ satisfies all requirements of (i)
since for any nonempty $C'\su C$, letting 
$\alpha'=\proj_{\R} C'$, we have 
$\beta'=\essuprime(\dim C'_t)=0$ and
$\dim C'\ge 2\alpha'$, which is strictly greater
than $\alpha'+\beta'$ if $\alpha'>0$.

To prove (ii) let 
$B=(C\times [0,1]) \cup ([0,1]\times\{(0,0,0)\})$.
Then clearly $B\su\R\times\R^3$ is a compact set
with 
$\al=\dim(\proj_{\R} B)=1$ and 
$\be = \essupno(\dim B_t)=1$
and it is non-$\Gamma_1$-null since it contains
the line segment $[0,1]\times\{(0,0,0)\}$.

Let $\emptyset\neq B' \su B$, 
$\al'=\dim(\proj_{\R} B')$ and
$\be' = \essuprime(\dim B'_t)$.
Using the dimension doubling property of $g$ and Theorem \ref{classic} we obtain that
if $\be'>0$ then $\dim B'\ge 2\al'+\beta'$,
which is strictly greater
than $\alpha'+\beta'$ if $\alpha'>0$.
\end{proof}
 
%
%

\begin{remark}
The reason that we have included $[0,1] \times \{(0,0,0)\}$ in the set $B$ is that $C \times [0,1]$ alone would be $\Ga_{1}$-null:
this can be seen either by using the second claim
of Corollary~\ref{fucor} or 
more directly, using the results of \cite{ABMP}. 
\end{remark}

Now we show that the Ahlfors--David regularity condition in Theorems \ref{thm1} and \ref{thm2} 
and in Corollary~\ref{fucor}
cannot be dropped: for example, the relaxed condition that 
$0 < \hal(\proj_{\rr^k} B) < \infty$ is not enough. 

\begin{prop}
\label{counter}
For any $0 < \al < 1$, there exists a compact set $B \su \rr^2$ with the following properties:
\begin{itemize}
	\item $B=A_1 \times A_2$ with $0 < \hal(A_i) < \infty$ ($i=1,2$), 
	
	\item for any $G \su B$ Borel $\gal$-null set we have $\dim (B \se G) = \al + 1$.
\end{itemize}
\end{prop}

Clearly, for the above $B$, 
the conclusions of 
Theorems \ref{thm1} and \ref{thm2} 
and Corollary~\ref{fucor}
do not hold.

\begin{proof}
The following construction can be found in \cite[Section 2.10.29]{Fe}.  

For any $0 < \al < 1$, there exist compact sets $A_1, A_2 \su \rr$ with $0 < \hal(A_i)<\infty$ ($i=1,2$) 
such that for $B=A_1 \times A_2$ the following is satisfied: 

For any $\ga < \al + 1$ and $\ep > 0$ there exists $\de>0$ (depending on $\al$, $\ep$ and $\ga$) such that for any $S \su \rr^2$ with $\diam S \leq \de$, we have 
\begin{equation}
(\hau^{\al} \times \hau^{\al})(B \cap S) \leq C \cdot \ep \cdot (\diam S)^{\ga}, 
\label{eqdiamtrick}
\end{equation}
where $C$ is a constant depending only on $\al$. 

We will check that \eqref{eqdiamtrick} implies that for every 
$B' \su B$ with $(\hau^{\al} \times \hau^{\al})(B')>0$, we have
$\dim B' = \alpha+1$.

First note that the upper bound $\alpha+1$ follows from the fact that $B'\subset A_1\times A_2 \subset A_1 \times \mathbb{R}$, and $A_1 \times \mathbb{R}$ has Hausdorff dimension $\alpha+1$.

Fix $\ga < \al + 1$, $\ep>0$, and consider a covering $B' \su \bigcup S_i$ with $\diam S_i \leq \de$, where $\de$ is as above. 
Then by \eqref{eqdiamtrick}, we have that 
$$\sum (\diam S_i)^{\ga} \geq \sum \frac{1}{C \cdot \ep} (\hau^{\al} \times \hau^{\al})(B \cap S_i) 
\geq \sum \frac{1}{C \cdot \ep} (\hau^{\al} \times \hau^{\al})(B' \cap S_i),$$ 
and since $B' \su \bigcup S_i$ and $\hau^{\al} \times \hau^{\al}$ is an outer measure, we have that 
$$\sum (\diam S_i)^{\ga} \geq \frac{1}{C \cdot \ep} (\hau^{\al} \times \hau^{\al})(B').$$ 
Taking infimum over the coverings it follows that $$\hau^\gamma(B')\ge \frac{1}{C \cdot \ep} (\hau^{\al} \times \hau^{\al})(B').$$
Therefore $\dim B'\ge \gamma$ for every $\gamma<\alpha+1$, hence $\dim B'= \alpha+1$.

To conclude the proof of Proposition \ref{counter}, we show that if $G$ is a Borel $\gal$-null set and $B'= B \se G$, 
then $(\hau^{\al} \times \hau^{\al})(B')>0$, and this is enough. 

Let $G$ be a Borel $\gal$-null set. Then we have $\hal(\{ t \in A_1: (t,x) \in G \})=0$ for each $x \in A_2$,  
considering constant (Lipschitz) functions in \eqref{gamma}. 
This immediately implies that for each $x \in A_2$, $\hal(\{t \in A_1: (t,x) \in B'\})=\hal(A_1)$ which is positive and finite.   
Using Fubini's theorem (see e.g. \cite[Theorem 2.6.2]{Fe}) for the Borel set $B'$ we obtain that 
$$(\hau^{\al} \times \hau^{\al})(B') = \int_{A_2} \hal(\{t \in A_1: (t,x) \in B'\}) d\hal(x) >0,$$
which completes the proof.
\end{proof}

\section{Proofs of the main results}
\label{pfthm2}

In this section we prove Theorems \ref{thm2} and \ref{thm1} subject to two lemmas, Lemma \ref{meas} and Lemma \ref{lem1}, 
which will be proved in Section \ref{pfmeas} and Section \ref{pflem1}, respectively. 
While Lemma \ref{meas} is merely a (highly non-trivial) technical lemma aimed to ensure measurability of sets in our proofs, 
the majority of the proof of Theorem \ref{thm2} can be found in the proof of Lemma \ref{lem1}. 

We start with  two observations. 
\begin{lemma} In Theorem \ref{thm2}, we can assume without loss of generality that 
\label{obs}
\begin{enumerate}[(i)]
	\item 
	\label{noganull}
	 $B$ is not $\gal$-null, and thus $\hal(\proj_{\rr^k} B)>0$,
	
	\item
	\label{betan}
	$\be < n$. 
	
\end{enumerate}
\end{lemma}

\begin{proof}
Clearly, Theorem \ref{thm2} trivially holds if $B$ is $\gal$-null, we can take $G=B$. 

If $\be=n$ in Theorem \ref{thm2}, then the statement follows trivially: Since $B \su T \times \rr^n$, we have that 
	$\dim B \leq \al + n$, and we can take $G=\emptyset$. 
	
\end{proof}


\begin{lemma}
\label{cup}
Let $B \su T \times \rr^n$ such that $\dim B_t \leq \be$ for all $t \in T$, and assume that $B$ is of the form $B=\cup_{i=1}^{\infty} B_i$. If Theorem \ref{thm2} holds for each $B_i$, $i=1,2,\dots$ (with $T$ and $\be$ fixed as above), then it also holds for $B$. 
\end{lemma}


\begin{proof}
Since $\dim B_t \leq \be$ for each $t \in T$, we have that $\dim (B_i)_t \leq \be$ for each $t\in T$, for each $i$. By Theorem \ref{thm2} applied for each $B_i$,  there exist Borel $\gal$-null sets $G_i \su B_i$ such that 
$\dim (B_i \se G_i) \leq \al + \be$. By Remark \ref{gacup}, $G=\bigcup_{i=1}^{\infty} G_i$ is Borel $\gal$-null, and  
$$\dim (B \se G)=\sup_i \,\dim (B_i \se G) \leq \sup_i \,\dim (B_i \se G_i) \leq \al + \be,$$ 
which completes the proof. 
\end{proof}

Now we introduce more simplifying assumptions in Theorem \ref{thm2}.

\begin{lemma}
\label{aa}
We can assume without loss of generality that 
\begin{enumerate}[(i)]
\item 
\label{a1}
$\proj_{\rr^n} B \su [0,1]^n$,

\item
\label{a2}
$\hal(T) \leq 1$. 	
\end{enumerate}

\end{lemma}

\begin{proof} \ 
\begin{enumerate}[(i)]  
\item
This is trivial using Lemma \ref{cup}, 
and that Hausdorff dimension is preserved under homotheties.

\item
It follows from Lemma \ref{cup} and Remark \ref{ahlregm}
 that we can assume that $\hal(T)< \infty$. 
Then we can replace $B$ with a homothetic copy $B'$ and $T$ with a homothetic copy $T'$ 
such that for $A'=\proj_{\rr^k} B'$, $A' \su T'$ and $\hal(T') \leq 1$. 

\end{enumerate}
\end{proof}

Fix a Borel set $B \su T \times [0,1]^n$, where $T \su \rr^k$ is an  Ahlfors--David $\al$-regular Borel set such that $\hal(T) \leq 1$. Let $\be \in (0,n)$ such that 
\begin{equation}
\label{smallerbeta}
\dim B_t < \be \ \text{for all} \ t \in A,
\end{equation} 
where $A=\proj_{\rr^k} B \su \rr^k$.  Clearly, we have $A \su T$. 
By Remark \ref{gacup} and \eqref{betan} of Lemma \ref{obs}, 
it is enough to show that there exists a Borel set $G \su B$ such that $G$ is $\Ga_{\al}$-null, and $\dim (B \setminus G) \leq \al + \be$. 

We use the following properties of $T \su \rr^k$. 
\begin{equation}
\hal(A) \leq \hal(T) \leq 1,
\label{tdef}
\end{equation}
and there exists $c>0$ such that for all $t \in T$, $0<r\leq 1$, 
\begin{equation}
c \cdot r^{\al} \leq \hau^{\al}(T \cap B(t,r)).
\label{ahl}
\end{equation}

For our convenience, we introduce a new measure in place of Hausdorff measure using translated copies of open dyadic cubes. 

Let 
\begin{equation}
U=\{(u_1,\dots,u_n) \in \rr^n, u_i \in \{0, 1/2\} \ (i=1,\dots,n)\}. 
\label{equ}
\end{equation}

For $u \in U$, and $l \geq 1$, let 

\begin{equation}
\mathcal{\ti{D}}_{u,l}=\left\{ \prod_{j=1}^n \left(\frac{m_j}{2^l},\frac{m_j+1}{2^l}\right): (m_1,\dots,m_n) \in u + \zz^n \right\},
\label{opendul}
\end{equation}

\begin{equation}
\mathcal{D}_l=\bigcup_{u \in U} \mathcal{\ti{D}}_{u,l} =\left\{ \prod_{j=1}^n \left(\frac{m_j}{2^l},\frac{m_j+1}{2^l}\right): (m_1,\dots,m_n) \in U + \zz^n \right\},
\label{opendl}
\end{equation}
the family of open cubes which are translates of $2^{-l}$-dyadic cubes by $u \cdot 2^{-l}$ for  some $u \in U$, and 
\begin{equation}
\mathcal{D}=\bigcup_{l=1}^{\infty} \mathcal{D}_l. 
\label{opend}
\end{equation}

For a set $X \su \rr^n$, $s \geq 0$ and $\de>0$, let
$$\ti{\net}^s_{\de}(X)=\inf \left\{ \sum_{i=1}^{\infty} (\diam(D_i))^s : X \su \bigcup_{i=1}^{\infty} D_i, \ D_i \in \mathcal{D}, \ \diam(D_i) \leq \de \right\}$$ and 
$\net^s(X)=\lim_{\de \to 0} \ti{\net}^s_{\de}(X)$. 

It is not hard to check that there exists a constant $b(n,s) \neq 0$ depending only on $n$ and $s$ such that for every $X \su \rr^n$, 
$\hau^s(X) \leq \net^s(X) \leq b(n,s) \hau^s(X)$.

We will use the following notation. 
For 
$$D=\left(\frac{m_1}{2^l},\frac{m_1+1}{2^l}\right) \times \left(\frac{m_2}{2^l},\frac{m_2+1}{2^l}\right) \times \dots \times  \left(\frac{m_n}{2^l},\frac{m_n+1}{2^l}\right) \in \mathcal{D},$$ let 
\begin{equation}
s_D=2^{-l} \ \text{and} \ x_D=\left(\frac{m_1}{2^l}, \frac{m_2}{2^l}, \dots, \frac{m_n}{2^l} \right)
\label{sqxq}
\end{equation}
denote the side length and the lower left corner of $D$.

We will restrict ourselves to small enough cubes.
For any $\ep>0$, let $l_0=l_0(\ep)$ be the smallest integer such that 
\begin{equation}
\sum_{l=l_0}^{\infty} 2^{-l\ep} \leq c \cdot \ep^2,
\label{lcond}
\end{equation}
where $c$ is from \eqref{ahl}. 
Let 
\begin{equation}
\ti{\mathcal{D}}=\ti{\mathcal{D}}_{\ep}=\{D \in \mathcal{D}: s_D \leq 2^{-l_0}\}.
\label{modd}
\end{equation}

We would like to find covers for the sections of $B$ consisting of cubes from $\ti{\mathcal{D}}$. 
By 
\eqref{smallerbeta} we have that for every $t \in A$ and every $\ep>0$ there exists a countable cover 
$$B_t \su \bigcup_{i \in I_t} D_i^t \ \text{such that} \ D_i^t \in \ti{\mathcal{D}} \ \text{and} \ \sum_{i \in I_t} \left(\diam(D_i^t)\right)^{\be} \leq \ep.$$ 
We need more: we need to find a cover with the above properties such that $t \mapsto \{D_i^t\}_{i \in I_t}$ is a ($\hal$-)measurable map on $A$. 
This will be the content of Lemma \ref{meas} below. First we need to introduce some definitions. Let

\begin{equation}
H=\mathcal{P}(\ti{\mathcal{D}}), \ \text{the space of all subsets of} \ \ti{\mathcal{D}},
\label{modh}
\end{equation}
equipped with the natural topology defined via considering $H$ as $\{0,1\}^{\ti{\mathcal{D}}}$ and taking the product topology. 
Clearly, $H$ is compact. 

For $h \in H$, let 
\begin{equation}
\u(h)=\bigcup_{D \in h} D,
\label{union}
\end{equation} 
and 
\begin{equation}
\s(h)=\sum_{D \in h} \diam(D)^{\be}.
\label{diam}
\end{equation}

Note that in the above definition, the integer $l_0$, and thus the space $H$ depends on $\ep$. 

The proof of the following lemma was communicated to us by Zolt\'an Vidny\'anszky. The proof will be postponed to Section \ref{pfmeas} and the Appendix.  
\begin{lemma}
\label{meas}
For any $\ep>0$, 
there exists a Borel set $F \su A$ with $\hal(F)=\hal(A)$ and a Borel measurable function 
$\phi: F \to H$ such that for each $t \in F$, $B_t \su \u(\phi(t))$, and $\s(\phi(t)) \leq \ep$. 
\end{lemma}

Fix $\ep>0$, and fix $F$ as well as a Borel measurable $\phi: F \to H$ obtained in Lemma \ref{meas} for $\ep$. 
Since $\hal(A \setminus F)=0$ and since $B$ and $F$ are Borel sets, we have that 
$B \setminus (F \times \rr^n)$ is Borel $\gal$-null, 
so our $\gal$-null set $G$ (to be constructed below) can include $B \setminus (F \times \rr^n)$. 
Thus we can assume without loss of generality that $\phi$ is defined on $A$.

Fix $\rho \geq 1$ (we will take $\rho \approx \log(1/\ep)$), and introduce the following notation: 
\begin{equation}
\garo=\{f: T \to  \rr^n: f \ \text{is} \ \rho\text{-Lipschitz}\}.
\label{garo}
\end{equation}

The proof of Theorem \ref{thm2} is based on the following lemma. 
\begin{lemma}
\label{lem1}
There exists a Borel set $\ti{B} \su B$ ($\ti{B}=\ti{B}_{\ep,\rho}$) with the following properties:
\begin{enumerate}
\item 
for all $\ga \in [\ep,1)$, $\hau_{\infty}^{\al+\be+\ga}(\ti{B}) \lkb_{n,k,\al,\be} \rho^{\al+\be+\ga} \cdot \ep$, 

\item 
for all $f \in \garo$, $\hal(\{t \in \rr^k: (t,f(t)) \in B \se \ti{B} \}) \lkb_{n,k} \ep$. 

\end{enumerate}

\end{lemma}

We postpone the proof of Lemma \ref{lem1} to Section \ref{pflem1}. 

Now we show how Lemma \ref{lem1} implies Theorem \ref{thm2}. 

\begin{proof}[Proof of Theorem~\ref{thm2}] 
For $m=1,2,\dots$ let $\ti{B}_m=\ti{B}_{2^{-m},m}$ constructed in Lemma \ref{lem1}. Then $\ti{B}_m$ has the following properties:
\begin{equation}
\label{l1}
\text{for all} \ \ga \in [2^{-m},1), \ \hau_{\infty}^{\al+\be+\ga}(\ti{B}_m) \lkb_{n,k,\al,\be} m^{\al+\be+\ga} \cdot 2^{-m}, 
\end{equation}

\begin{equation}
\label{l2}
\text{for all} \ f \in \mathcal{G}_m, \ \hal(\{t \in \rr^k: (t,f(t)) \in B \se \ti{B}_m \}) \lkb_{n,k}  2^{-m}. 
\end{equation}

Let 
$$B'=\limsup \ti{B}_m=\bigcap_{l=1}^{\infty} \bigcup_{m=l}^{\infty} \ti{B}_m, \ \text{and} \ G=B \se B'.$$ 
Then clearly, $G$ is Borel. 
We claim that $G$ is $\gal$-null (see the definition in \eqref{gamma}), and $\dim (B \setminus G) \leq \al + \be$. 

First we show that $\dim B'=\dim (B \setminus G) \leq \al + \be$. 

Let $0 < \ga \leq 1$ be arbitrary. Let $l_1$ be the first integer such that $2^{-l_1} \leq \ga$. Then 
 we have by \eqref{l1} that 
\begin{equation}
\label{eqf0}
\hau_{\infty}^{\al+\be+\ga}(B') \leq \hau_{\infty}^{\al+\be+\ga}(\bigcap_{l=l_1}^{\infty} \bigcup_{m=l}^{\infty} \ti{B}_m) \leq 
\sum_{m=l}^{\infty} \hau_{\infty}^{\al+\be+\ga}(\ti{B}_m) \lkb \sum_{m=l}^{\infty} m^{\al + \be + \ga} \cdot 2^{-m}
\end{equation}
for each $l \geq l_1$. 
Since $\sum m^{\al + \be+\ga} 2^{-m}$ is convergent, the right hand side tends to zero as $l \to \infty$, so we obtain that 
$\hau_{\infty}^{\al+\be+\ga}(B')=0$ for each $0 < \ga \leq 1$. Clearly, this implies that 
$\dim B' \leq \al + \be$, which completes the proof of our first claim. 

Now we prove that for every  Lipschitz function $f: T \to \rr^n$, 
$\hal(\{t \in \rr^k: (t,f(t)) \in G \}=0$.

Let $f$ be $\kappa$-Lipschitz for some $\kappa$. 
We need to prove that 
\begin{equation}
\hal(\{t \in \rr^k: (t,f(t)) \in \bigcup_{l=1}^{\infty} \bigcap_{m=l}^{\infty} (B \se \ti{B}_m) \})=0.
\label{eqf2}
\end{equation}

Clearly, it is enough to prove that for each $l=1,2,\dots,$ 
\begin{equation}
\hal(\{t \in \rr^k: (t,f(t)) \in \bigcap_{m=l}^{\infty} (B \se \ti{B}_m) \})=0. 
\label{eqf3}
\end{equation}

Let $l_2$ be the first integer such that $\kappa \leq l_2$. By \eqref{l2}, we have that 
$\hal(\{t \in \rr^k: (t,f(t)) \in B \se \ti{B}_m \}) \lkb 2^{-m}$ for all $m \geq l_2$. Taking $m\to\infty$ we obtain \eqref{eqf3}, which completes the proof. 
\end{proof}

Now we combine Theorem \ref{classic} with Theorem \ref{thm2} to prove Theorem \ref{thm1}.

\begin{proof}[Proof of Theorem~\ref{thm1}]
Let $B \su  \rr^k \times \rr^n$ be a Borel set such that $\dim \proj_{\rr^k} B = \al>0$ and $\proj_{\rr^k} B \su T$, 
where $T \su \rr^k$ is an Ahlfors--David $\al$-regular set. 
Let $A=\proj_{\rr^k} B$, and let $\ga=\essup(\dim B_t)$.
By Remark \ref{true}, $\dim B \geq \al + \ga$. Using Howroyd's theorem (see \cite{Ho}), we can choose a Borel subset 
\begin{equation}
B_1 \su B \ \text{with} \ \dim B_1=\al + \ga.
\label{eqhowy}
\end{equation} 

If $\hal(A)=0$, then $B$ is $\gal$-null, therefore $B \setminus B_1$ is $\gal$-null as well. Choosing $G=B \setminus B_1$, the statement concludes. 

Now we assume that $\hal(A)>0$. 
Let $\ep>0$. 
By 
Definition~\ref{d:ess-sup},
$\hal (\{t \in A: \dim B_t > \ga + \ep \})=0$. 
We let $$A'=A'_{\ep}=\{t \in A: \dim B_t \leq \ga + \ep \}.$$ 
We have $\hal(A \se A')=0$. Moreover, $A'$ is $\hal$-measurable, 
since it is in the $\sigma$-algebra generated by analytic sets (see \cite{De} or \cite[Remark 6.2.]{MM}). 
This, combined with the fact that $T$ is a Borel set with $\sigma$-finite $\hal$-measure and $A' \su T$, 
easily implies that we can find a Borel set $A'' \su A'$ with 
$\hau^{\al} (A' \se A'')=0$. Then
$\hal (A \se A'')=0$, and $G'_{\ep}=B \setminus (A'' \times \rr^n)$ is a Borel $\gal$-null set. 

On the other hand, $A''$ was constructed so that we can apply Theorem \ref{thm2} for 
$B''=(A'' \times \rr^n) \cap B$ and $\be=\ga + \ep$. 
We conclude that there exists $G_{\ep}'' \su B''$ such that 
$G_{\ep}''$ is $\Ga_{\al}$-null, and 
$\dim (B'' \se G_{\ep}'') \leq \al + \ga + \ep$. Let
$G_1=\bigcup_{\ep \in \mathbb{Q}^+} G''_{\ep} \cup G'_{\ep} $. Clearly, $\dim (B \se G_1) \leq \al + \ga$  
and  by Remark \ref{gacup}, $G_1$ is $\gal$-null. 
Finally, let 
$G=G_1 \se B_1$, where $B_1$ is from \eqref{eqhowy}. Then $G\su B$ is a Borel $\Ga_{\al}$-null set, and $\dim (B \setminus G) = \al + \ga$, 
which is precisely the conclusion of Theorem \ref{thm1}. 
\end{proof}

\section{The proof of Lemma \ref{meas}}
\label{pfmeas}
The following more general statement will be shown. Part of its proof involves powerful techniques from set theory, these will be postponed to the Appendix. In case $B$ is compact, the proof becomes much easier, see Remark \ref{meascompact} below. 
\begin{lemma}[Z. Vidny\'anszky]
\label{genmeas}
Let $T \su \rr^k$ Borel with $0< \hal(T) < \infty$, let $\ep>0$, let $B \su T \times \rr^n$ Borel, and denote $A=\proj_{\rr^k} B\su T$. 
Suppose that for all $t \in A$, $\hau^{\be}(B_t)=0$. 
Then 
there exists a Borel set $F \su A$ with $\hal(F)=\hal(A)$ and a Borel measurable function 
$\phi: F \to H$ such that for each $t \in F$, $B_t \su \u(\phi(t))$, and $\s(\phi(t)) \leq \ep$ (see \eqref{union} and \eqref{diam}). 
\end{lemma}

First we show that the statement of Lemma \ref{genmeas} is consistent with the standard ZFC axioms of set theory (so it cannot be disproved in ZFC). This will be done by showing that \eqref{*assume} below implies Lemma \ref{genmeas}, and using the fact that \eqref{*assume} is consistent with ZFC, 
see \cite[Corollary 9.3.2]{BJ}. 

\begin{equation}
\label{*assume}
\text{ If } F \su [0,1] \text{ is the projection of a co-analytic set then  $F$ is Lebesgue measurable.}
\end{equation}


\begin{prop}
\label{*prop}
Lemma \ref{genmeas} follows from \eqref{*assume}. Consequently,  Lemma \ref{genmeas} is consistent with ZFC. 
\end{prop}
Unfortunately, \eqref{*assume} cannot be proved in ZFC alone, see \cite[p. 307]{Kec}.  
Nevertheless, using the Shoenfield absoluteness theorem (\cite[Theorem 25.20]{Je}) and Proposition \ref{*prop} we are also able to prove Lemma \ref{genmeas} in ZFC. Since this requires very involved techniques from set theory and logic, the proof will be postponed to the Appendix. 

\begin{proof}
Assume that \eqref{*assume} holds. 
Using the isomorphism theorem for measures (see e.g. \cite[Theorem 17.41.]{Kec}) and the simple fact that if a set is a projection 
of a co-analytic set then any Borel isomorphic image of it is also a projection of a co-analytic set,  \eqref{*assume} implies the following. 
\begin{align}
\label{*Hausdorff_assume}
 & \text{For any} \ T \su \rr^k \ \text{Borel with} \ 0< \hal(T)  < \infty, \\
 & \text{if $F \su T$ is the projection of a co-analytic set then  $F$ is }  \hal \text{-measurable.}  
\nonumber
\end{align}

Let 
$$C=\{(t,h) \in T \times H \ : \ B_t \su \u(h), \ \s(h) \leq \ep\}.$$ 
Since $\hau^{\be}(B_t)=0$ for $t \in A$ and $B_t = \emptyset$ for $t \in T \setminus A$, we have that $\proj_T C = T$. 
Our goal is to find an $\hal$-measurable uniformization of $C$, that is, an $\hal$-measurable function 
$\psi: T \to H$ such that $\graph(\psi) \su C$: if $\psi$ is such a function, then by definition 
for each $t \in T$ we have $B_t \su \u(\psi(t))$, and $\s(\psi(t)) \leq \ep$.

We claim that $C$ is a co-analytic set. 
Let 
$$
L=\{((t,y),h) \in B  \times H: y \notin \u(h)\}
\quad \text{and} \quad
S=\{h \in H: \s(h) \leq \ep\}.
$$ 
Then, by definition, 
$$
C=
\left((T \times H) \se \proj_{T \times H} L\right) 
\cap (T\times S).$$
It is easy to check using that the cubes of the union $\u(h)$ are open that $(B \times H) \se L$ is open in $B \times H$. Since $B$ is Borel, this implies that $\proj_{T \times H} L$ is analytic. 
It is also easy  to check that $S$ is closed, so $T \times S$ is Borel, hence the above formula gives that $C$ is indeed 
co-analytic.

By the Novikov-Kondo uniformization theorem (see \cite[Theorem 36.14]{Kec}), $C$ has a co-analytic uniformization,  
that is, there exists $\psi: T \to H$ such that $\graph(\psi) \su C$, and $\graph(\psi)$ is co-analytic. 
We need that $\psi$ is $\hal$-measurable. Since for any open $U \su H$, $\psi^{-1}(U)=\proj_T \left((T \times U) \cap \graph(\psi)\right)$, 
and $(T \times U) \cap \graph(\psi)$ is co-analytic, $\psi^{-1}(U)$ is $\hal$-measurable by \eqref{*Hausdorff_assume}. 

Recall that $A=\proj_{\rr^k} B \su T$. Since $B$ is Borel, $A$ is analytic, and since analytic sets are measurable, we have that 
$\psi \restriction A$ is $\hal$-measurable. 
Using Lusin's theorem (see e.g. \cite[Corollary 2.3.6]{Fe}), the fact that $\hal(A)<\infty$, and the regularity of the $\hal$-measure, we get that there exists $F \su A$ Borel with 
$\hal(F)=\hal(A)$ and a Borel-measurable function $\phi: F \to H$ such that $\phi(t)=\psi(t)$ for every $t \in F$, 
concluding the proof of Proposition \ref{*prop}.

\end{proof}

\begin{remark}
\label{meascompact}
If  $B$ is compact, then the proof of Lemma~\ref{meas} is much simpler. Note that each covering of a section of $B$ by open (dyadic) cubes has a subcover with finitely many cubes, so we can restrict our attention to $h\in H$ with $|h|<\infty$.
Let 
$$C'=\{(t,h) \in A \times H: B_t \su \u(h), \ \s(h) \leq \ep, \ |h| < \infty\}.$$ 

Let 
$$L=\{((t,y),h) \in B  \times H: y \notin \u(h)\} \ \text{and} \  S=\{h \in H: \s(h) \leq \ep\}$$ 
as above. It is easy to check using that $B$ is compact that $L$ is compact. 
Also, $S$ is closed, so $A \times S$ is Borel. 
Let $G=\{h \in H:  |h| < \infty\}$.
Clearly, $G$ is countable, so $A \times G$ is Borel. By definition, 

\begin{align}
C'=\left( (A \times H) \se \proj_{T \times H} L \right) \cap (A \times S) \cap (A \times G). 
\end{align}

Clearly, $C'$ is Borel, moreover, for all $t \in A$, the $t$-section $\{h \in H: (t,h) \in C'\}$ is nonempty and countable. 
We can apply the Lusin-Novikov uniformization theorem (see \cite[Theorem 18.10]{Kec}),
 which implies that there is a Borel measurable function $\phi: A \to H$ such that $\gr \phi \su C'$. This means that 
there exists a Borel measurable function $\phi: A \to H$ such that for each $t \in A$, $B_t \su \u(\phi(t))$, and $\s(\phi(t)) \leq \ep$, 
which completes the proof of Lemma \ref{meas} in case $B$ is compact.  
\end{remark}

\begin{remark}
For an interesting related result, see \cite{FaMau}, where the authors also use descriptive set theory to derive 
Fubini-type results (of a different perspective) for general measures. 
\end{remark}

\section{The proof of Lemma \ref{lem1}}
\label{pflem1}

We introduce some more notation. 

For $t \in A$, let $\phi(t)=h_t \in H$.  
By Lemma \ref{meas}, we have that for all $t \in A$, 
\begin{equation}
B_t \su \bigcup_{D \in h_t} D, \  \sum_{D \in h_t} \diam(D)^{\be} \leq \ep,
\label{m1}
\end{equation}
and $\phi: t \mapsto h_t$ is Borel measurable. 

Fix $t \in A$. For $D \in h_t$, let
\begin{equation}
\label{ball}
Q_D=\{t\} \times D \ \text{and} \ \iq^t=\{Q_D: D \in h_t\},
\end{equation}
which is a countable collection of (relative) open $n$-dimensional cubes in $\rr^{k+n}$. 

For $Q=\{t\} \times D \in \iq^t$, let $s_Q=s_D$ and $x_Q=x_D$ (recall \eqref{sqxq}). 
For $t \in A \su \rr^k$ and $Q \in \iq^t$, let 
\begin{equation}
C_Q=\ol{B}(t,s_Q) \times (x_Q+ 3\rho \cdot s_Q \cdot [-1,1]^n) \su \rr^{k+n},
\label{cilin}
\end{equation}
which is a closed cylinder in $\rr^{k+n}$ containing $Q$ in its interior. 

Let 
$$\ir^t=\{C_Q: Q \in \iq^t\} \ \text{and} \ 
\ir=\{C_Q: Q \in \iq^t, t \in A\}.$$

Then by definition, it is easy to check that 
\begin{equation}
\label{lindel}
B \su \bigcup_{C \in \ir} \bel (C).
\end{equation}

Intuitively, our aim is to choose a suitable subset $\ti{B}$ of $B$, for which 
$$\sum_{C \in \ir, C \cap \ti{B} \neq \emptyset} \diam(C)^{\al + \be}$$ is well-comparable with 
$$\int_{t \in A} \sum_{Q \in \iq^t} \diam(Q)^{\be} \ d\hau^{\al}(t).$$ 
Definition \ref{defci} below will lead to the definition of such a subset $\ti{B}$. 

For $C=C_Q \in \ir$, let $r_C=s_Q \in \rr$, $x_C=x_Q \in \rr^n$, and $y_C=t \in A$. 

\begin{defin}
\label{defci}
We say that $C \in \ir$ is \emph{good}, if 
\begin{equation}
r_C^{\al + \be+\eps} \le \int\limits_{A \cap \ol{B}(y_C,r_C)} \sum_{Q \in \iq^t, Q \su C, } {s_Q}^{\be} \ d\hau^{\al}(t).
\label{goodci}
\end{equation}
\end{defin}

\begin{remark}
\label{remeas}
It is easy to check that Lemma \ref{meas} implies that the function 
$t \mapsto \sum_{Q \in \iq^t, Q \su C, } {s_Q}^{\be}$ is 
$\hal$-measurable, 
and bounded by $\ep$ for all $t \in A$, thus $\hau^{\al}$-integrable on $A$. 
\end{remark}

Let 
\begin{equation}
\ic=\{C \in \ir: C \ \text{is \emph{good}} \},
\label{ic}
\end{equation}

and 
\begin{equation}
\ti{B}=B \cap \bigcup_{C \in \ic}\bel{C}.
\label{tib}
\end{equation}

We claim that the statements in Lemma \ref{lem1} are true for $\ti{B}$ defined above. 
Clearly, $\ti{B}$ is Borel. 

First we show that  
$$ \forall \ \ga \in [\ep,1),  \ \hau_{\infty}^{\al+\be+\ga}(\ti{B}) \lkb_{n,k,\al,\be} \rho^{\al+\be+\ga} \cdot \ep.$$ 
Fix $\ga \in [\ep,1)$. 

It is easy to see that each $C \in \ic$ is a closed ball of radius $r_C$ in the metric space 
$(\rr^{k+n},d)$, where $d=\max\{d_k, 1/(3\rho) \cdot m_n\}$, where $d_k$ denotes the Euclidean metric on $\rr^k$, and $m_n$ denotes the maximum metric on $\rr^n$. 
We can apply the ``$5r$''-covering theorem (see e.g. \cite[Theorem 2.1]{Ma}) for the collection $\ic$ of balls to obtain the following: 
There exists $\ic' \su \ic$ such that $\ic'$ is a disjoint family ($C \cap C' = \emptyset$ for all $C \neq C' \in \ic'$), and 
such that 
$$\ti{B} \su \bigcup_{C \in \ic'} 5C,$$
where $5C$ denotes the ball which is centered at the same center as $C$, and has radius $5$-times the radius of $C$. 

Then by the definition of $\hal$ (using now the standard Euclidean metric on $\rr^{k+n}$), it is easy to see by $\ga \leq 1$ that  
\begin{equation}
\hau^{\al+\be+\ga}_{\infty}(\ti{B}) \lkb_{n,k,\al,\be} \sum_{C\in \ic'} (5 \rho \cdot r_C)^{\al+\be+\ga} \lkb 
\rho^{\al+\be+\ga} \cdot \sum_{C\in \ic'} r_C^{\al+\be+\ep} \cdot r_C^{\ga-\ep}.
\label{eqe0}
\end{equation}

Since each $C \in \ic'$ is good,  by definition we have that 
$$\sum_{C\in \ic'} r_C^{\al+\be+\ep} \leq  \sum_{C\in \ic'} 
\int\limits_{A \cap \ol{B}(y_C,r_C)} \sum_{Q \in \iq^t, Q \su C, } {s_Q}^{\be} \ d\hau^{\al}(t).$$
By the trivial containment $A \cap \ol{B}(y_C,r_C) \su A$, using that $\ic'$ is a disjoint family, and 
\eqref{tdef} we obtain that this is at most 
$$\sum_{C\in \ic'} \int\limits_{A} \sum_{Q \in \iq^t, Q \su C, } {s_Q}^{\be} \ d\hau^{\al}(t) 
\leq \int\limits_{A} \sum_{Q \in \iq^t} {s_Q}^{\be} \ d\hau^{\al}(t) \leq 
\int\limits_{A} \ep d\hau^{\al}(t) \leq \ep.$$ 
Since $\ga \geq \ep$ and $r_C \leq 1$, we have $r_C^{\ga-\ep}\leq 1$ for each $C$. 
Continuing the estimate in \eqref{eqe0}, we obtain that 
$$\hau^{\al+\be+\ga}_{\infty}(\ti{B}) \lkb_{n,k,\al,\be} \rho^{\al+\be+\ga} \ep,$$ 
which proves our first claim in Lemma \ref{lem1}.

Now we prove that for all $f \in \garo$, $\hal(\{t \in \rr^k: (t,f(t)) \in B \se \ti{B} \}) \lkb_{n,k} \ep$, where $\garo$ is defined in \eqref{garo}. 

Recall the definition of $\iq^t$ from \eqref{ball}, $U$ from \eqref{equ}, and $l_0$ from \eqref{lcond}. 
For $t \in A$, $u \in U$, and  $l \geq l_0$, let 
\begin{equation}
\iq^t_{u,l}=\{Q \in \iq^t: s_Q=2^{-l}, \ 
x_Q=\left(\frac{m_1}{2^l}, \frac{m_2}{2^l}, \dots, \frac{m_n}{2^l} \right) \ \text{with} \ (m_1,\dots,m_n) \in u + \zz^n \}.
\label{iqtul}
\end{equation}

Note that by definition, for each $t \in A$, 
$$\{t\} \times B_t \su \bigcup_{u \in U} \bigcup_{l \geq l_0} \bigcup_{Q \in \iq^t_{u,l}} Q.$$

\begin{remark}
\label{disjoint}
We can assume without loss of generality, that for each $t \in A$ and $u \in U$, if $Q_1 \in \iq^t_{u,l_1}$ and $Q_2 \in \iq^t_{u,l_2}$ 
such that $l_1 \neq l_2$, then $Q_1 \cap Q_2 = \emptyset$. 
\end{remark}

Fix $f \in \garo$. Since for each fixed $u \in U$ and $l \geq l_0$, 
$\iq^t_{u,l}$ is a disjoint collection for any $t \in A$, there exists at most one $Q \in \iq^t_{u,l}$ such that $(t,f(t)) \in Q$. 

For $u \in U$ and $l \geq l_0$, let 
\begin{equation}
A_{u,l}=\{t \in A: (t,f(t)) \in \bigcup_{Q \in \iq^t_{u,l}} Q\}, 
\label{aul}
\end{equation}
and for any $t \in A_{u,l}$,
\begin{equation}
\text{let} \ Q_t \ \text{denote the unique cube in} \ \iq^t_{u,l} \ \text{such that} \ (t,f(t)) \in Q_t.
\label{qt}
\end{equation}

For any $t \in A_{u,l}$ write $C_t$ for the cylinder generated from $Q_t$ as in \eqref{cilin}. That is, 
\begin{equation}
C_t=\ol{B}(t,2^{-l}) \times (x_{Q_t}+ 3\rho \cdot 2^{-l} \cdot [-1,1]^n) \su \rr^{k+n}. 
\label{ct}
\end{equation}
We will need the following easy geometrical lemma. 

\begin{lemma}
\label{local}
For any $u \in U$ and $l \geq l_0$, if $y \in A_{u,l}$ and $t \in \ol{B}(y,2^{-l}) \cap A_{u,l}$, then $Q_t \su C_y$.
\end{lemma}

\begin{proof}
Let $Q_t=\{t\} \times D_t$, and let $v \in D_t$ be arbitrary. Then by definition, 
\begin{equation}
|v_j-f(t)_j| < 2^{-l}
\label{v1}
\end{equation}
for each coordinate $j=1,\dots,n$. 

Moreover, by the definition of $\garo$ (see \eqref{garo}), and by $t \in \ol{B}(y,2^{-l})$, 
\begin{equation}
|f(t)_j-f(y)_j| \leq |f(t)-f(y)| \leq \rho |t-y| \leq \rho \cdot 2^{-l} 
\label{v2}
\end{equation}
for each $j=1,\dots,n$. 

Let $Q_y=\{y\} \times D_y$, and $D_y=z_{y}+2^{-l} \cdot (0,1)^n$. Then we have 
\begin{equation}
|z_{y,j}-f(y)_j| < 2^{-l}
\label{v3}
\end{equation}
for each $j=1,\dots,n$.

Combining \eqref{v1},  \eqref{v2}, and \eqref{v3}, and using  $\rho \geq 1$ we obtain that 
$$|v_j-z_{y,j}| < 2^{-l} + \rho \cdot 2^{-l} + 2^{-l} \leq 3 \cdot \rho \cdot  2^{-l}$$
for each $j=1,\dots,n$. 

Then $v$ is contained in the cube $z_{y}+3 \rho 2^{-l} [-1,1]^n$. 
This immediately implies by the definition of $C_y$ (see \eqref{ct}) that $(t,v)$ is contained in $C_y$, so $Q_t \su C_y$, which completes the proof  of Lemma \ref{local}. 

\end{proof}

By Remark \ref{disjoint}, for any $u \in U$, 
\begin{equation}
A_{u,l_1} \cap A_{u,l_2} = \emptyset \ \text{for all} \ l_1 \neq l_2.
\label{auldisj}
\end{equation}

For each $u \in U$, let
$$R_u=\bigcup_{l \geq l_0}A_{u,l} \su A.$$

\begin{remark}
\label{aulmeas}
It is easy to check using Lemma \ref{meas} that $A_{u,l}$ is 
$\hal$-measurable for all $l \geq l_0$. 
\end{remark}

Let $a_{u,l}=\hal(A_{u,l})$ for $l \geq l_0$. By \eqref{auldisj}, Remark \ref{aulmeas}, and \eqref{a2} of Lemma \ref{aa}, we have that 
\begin{equation}
\sum_{l=l_0}^{\infty} a_{u,l} = \hal(R_u) \leq \hal(A) \leq 1 
\label{aulsum}
\end{equation}
for each $u \in U$. 

Fix $u \in U$, and for simplicity, let $A_l=A_{u,l}, a_l=a_{u,l}$ for $l \geq l_0$. Recall the properties of $T$ from \eqref{tdef} and \eqref{ahl}. 

We define 
\begin{equation}
A_l'=\{y \in A_l: \frac{\hal(A_l \cap \ol{B}(y,2^{-l}))}{\hal(T \cap \ol{B}(y,2^{-l}))} \leq \ep \cdot a_l\},
\label{ald}
\end{equation}
the subset of $A_l$ consisting of points which are centers of such closed balls of radius $2^{-l}$ in which the density of $A_l$ is at most 
$\ep \cdot a_l$.

\begin{lemma}
\label{alclaim}
We have 
$
\hal(A_l') \lkb_k \ep \cdot a_l$ for each $l \geq l_0$.
\end{lemma}

\begin{proof}
We apply Besicovitch's covering theorem (see e.g. \cite[Theorem 2.7]{Ma}) for the collection $\ib_l$ of closed balls of radius $2^{-l}$ 
centered at the points of $A_l'$. We obtain that there exists a constant $c(k)$ depending only on $k$ and collections of balls $B_1,\dots, B_{c(k)} \su \ib_l$ such that 
$$A_l' \su \bigcup_{i=1}^{c(k)} \bigcup_{B \in B_i} B,$$ 
and $B_i$ is a disjoint collection for each $i$, that is, 
$$B \cap B'=\emptyset \ \text{for all} \ B, B \in B_i, B \neq B'.$$
Then we have by \eqref{ald} that 
$$\hal(A_l') \leq \sum_{i=1}^{c(k)} \hal(\bigcup_{B \in B_i} (B \cap A_l)) \leq \sum_{i=1}^{c(k)}  \ep \cdot a_l \cdot \sum_{B \in B_i} \hal(B \cap T).$$ 
Using that each $B_i$ is a disjoint collection, $T$ is Borel, and \eqref{tdef}, we obtain that 
$$\hal(A_l') \leq \sum_{i=1}^{c(k)} \ep \cdot a_l \cdot \hal(\bigcup_{B \in B_i} (B \cap T)) \leq c(k) \cdot \ep  \cdot a_l \cdot \hal(T)
\lkb_k \ep \cdot a_l$$ 
and the proof of Lemma \ref{alclaim} concludes. 
\end{proof}

Let 
\begin{equation}
A_l''=A_l \se A_l'=\{y \in A_l: \frac{\hal(A_l \cap \ol{B}(y,2^{-l}))}{\hal(T \cap \ol{B}(y,2^{-l}))} > \ep \cdot a_l\},
\label{alc}
\end{equation}

\begin{equation}
\ij=\{l \geq l_0: a_l \geq \frac{2^{-l\ep}}{c \cdot\ep}\}, \ \text{and} \  
\ij'=\{l \geq l_0: a_l < \frac{2^{-l\ep}}{c \cdot\ep}\}, 
\label{ij}
\end{equation}
where $c$ is from \eqref{ahl}.

We prove the following lemma. 
\begin{lemma}
\label{clco}
If $l \in \ij$ and  $y \in A_l''$, then $C_y \in \ic$, where $C_y$ is from \eqref{ct}, and $\ic$ is defined in \eqref{ic}. 
\end{lemma}

\begin{proof}
By Lemma \ref{local}, we have that 
\begin{align}
\label{z1}
\int_{\ol{B}(y,2^{-l}) \cap A} \sum_{Q \in \iq^t, Q \su C_y, } {s_Q}^{\be} \ d\hau^{\al}(t) & 
\geq \int_{\ol{B}(y,2^{-l}) \cap A_l} \sum_{Q \in \iq^t, Q \su C_y, } {s_Q}^{\be} \ d\hau^{\al}(t) \geq \nonumber \\
& \geq \int_{\ol{B}(y,2^{-l}) \cap A_l} {s_{Q_t}}^{\be} \ d\hau^{\al}(t) = \nonumber \\
& = \hal(\ol{B}(y,2^{-l}) \cap A_l) \cdot 2^{-l\be}.
\end{align}

Using $y \in A_l''$, \eqref{ahl}, and $l \in \ij$, we have that 
\begin{align}
\label{z2}
\hal(\ol{B}(y,2^{-l}) \cap A_l) & > \ep \cdot a_l \cdot \hal(\ol{B}(y,2^{-l}) \cap T) \geq \ep \cdot a_l \cdot c \cdot 2^{-l\al} \geq \nonumber \\
& \geq \ep \cdot \frac{2^{-l\ep}}{c \cdot \ep} \cdot c \cdot 2^{-l\al}=2^{-l(\ep + \al)},
\end{align}
where $c$ is from \eqref{ahl}. 
Combining \eqref{z1} and \eqref{z2} we obtain that 
$$\int_{\ol{B}(y,2^{-l}) \cap A} \sum_{Q \in \iq^t, Q \su C_y, } {s_Q}^{\be} \ d\hau^{\al}(t)  \geq 2^{-l(\ep + \al+\be)}.$$ 
Comparing with the definition in \eqref{goodci}, we conclude that $C_y$ is good, which completes the proof of Lemma \ref{clco}.  
\end{proof}
Now we continue with the proof of Lemma \ref{lem1}. 

Let $S=\{t \in \rr^k: (t,f(t)) \in B \se \ti{B} \}$, where $\ti{B}=B \cap \left(\bigcup_{C \in \ic} \bel{C}\right)$ defined in \eqref{tib}. 
We need to show that $\hal(S) \lkb_{n,k} \ep$. 

Clearly, 
\begin{equation}
S \su  \bigcup_{u \in U} \bigcup_{l \geq l_0} A_{u,l} = \bigcup_{u \in U} \left(\bigcup_{l  \in \ij} A_{u,l} \cup \bigcup_{l  \in \ij'} A_{u,l} \right).
\label{w1}
\end{equation}

In fact, by the definition of $\ti{B}$ and Lemma \ref{clco}, we have that 
$$S \su \bigcup_{u \in U} \left(\bigcup_{l  \in \ij} A_{u,l}' \cup  \bigcup_{l  \in \ij'} A_{u,l}  \right),
$$
thus
\begin{equation}
\hal(S) \leq \sum_{u \in U} \left(\sum_{l \in \ij} \hal(A_{u,l}') + \sum_{l \in \ij'} \hal(A_{u,l}) \right).
\label{w3}
\end{equation}

For any fixed $u$, recall the definition of $\ij'$ from \eqref{ij} (which depends on $u$ since $a_l=a_{u,l}$). We conclude from \eqref{ij} and \eqref{lcond} that 
\begin{equation}
\sum_{l \in \ij'} \hal(A_{u,l})=\sum_{l \in \ij'} a_{u,l} \leq \sum_{l=l_0}^{\infty} \frac{2^{-l\ep}}{c \cdot\ep} \leq \ep.
\label{w4}
\end{equation}

By Lemma \ref{alclaim} and \eqref{aulsum}, we have  
\begin{equation}
\sum_{l \in \ij} \hal(A_{u,l}')  \lkb_k  \sum_{l \in \ij} \ep \cdot a_{u,l} \leq \ep. 
\label{w5}
\end{equation}

Combining \eqref{w3}, \eqref{w4} and \eqref{w5}, we obtain that 
$$\hal(S) \lkb_{k} |U| \cdot \ep=2^n \cdot\ep \lkb_{n,k}  \ep,$$ 
which finishes the proof of Lemma \ref{lem1}.

\section*{Acknowledgment}
The authors are grateful to Rich\'ard Balka for the helpful conversation related to Brownian motion in Proposition~\ref{p:Brownian}. 
We are also grateful to M\'arton Elekes, Lajos Soukup, and  Zolt\'an Vidny\'anszky for providing insight into the set theoretical aspects of Lemma \ref{meas}. 
We are especially grateful to Zolt\'an Vidny\'anszky for writing the proof in the Appendix. 
The authors also thank the anonymous referee for the careful reading and helpful comments. 

\appendix
\label{app}
\section{The ZFC proof of Lemma~\ref{genmeas}}
We show that the statement of Lemma \ref{genmeas} holds in ZFC. For the convenience of the reader, let us recall it: 
	\begin{align}
	\label{eqpi13Zoli1}
	& 
	\forall \ T \subset \R^k \ \text{Borel}, \ \forall \ \epsilon>0, \ \forall \ B \subset T \times \R^n \ \text{Borel,} \\ \nonumber
	& (0=\Hausdorff{\alpha}(T) \ \text{or} \ \Hausdorff{\alpha}(T)=\infty \ \text{or} \ \exists \ t \in \proj_{\R^k}(B) \ \text{such that} \ \Hausdorff{\beta}(B_t)>0), \\ \nonumber
	& \text{or } \ (\exists \ F \subset A \ \text{Borel  and} \ \exists \ 
	\phi: F \to H \ \text{Borel measurable such that} \\ \nonumber
	& (\Hausdorff{\alpha}(F)=\Hausdorff{\alpha}(A) \ \text{and} \ \forall t \in F, B_t \subset \mathcal{U}(\phi(t)) \ \text{and} \ \mathcal{S}(\phi(t))) \leq \epsilon )) . 
	\end{align}
	
	We check that the above statement is equivalent to the following, which will be more convenient for us to use:
	\begin{align}
	\label{eqpi13Zoli2}
	& 
	\forall \ T \subset \R^k \ \text{Borel}, \ \forall \ \epsilon>0, \ \forall \ B \subset T \times \R^n \ \text{Borel,} \\ \nonumber
	& 0=\Hausdorff{\alpha}(T) \ \text{or} \ \Hausdorff{\alpha}(T)=\infty \ \text{or} \ \exists \ t  \ \text{such that} \ \Hausdorff{\beta}(B_t)>0 \\ \nonumber
	& \text{or } \ \exists \ F \subset T \ \text{Borel  and} \ \exists \ 
	\phi: \R^k \to H \ \text{Borel measurable such that} \\ \nonumber
	& \Hausdorff{\alpha}(T \setminus F)=0 \ \text{and} \ \forall t \in F, B_t \subset \mathcal{U}(\phi(t)) \ \text{and} \ \mathcal{S}(\phi(t) \leq \epsilon . 
	\end{align}
To see that \eqref{eqpi13Zoli1} implies \eqref{eqpi13Zoli2}: let $F \su  A$ with $\hal(F)=\hal(A)$ as in \eqref{eqpi13Zoli1}. Using that  $\hal(T)<\infty$, choosing $U \su T \setminus A$ Borel with $\hal(U)=\hal(T \setminus A)$ and setting $\ti{F}=F \cup U$, we have that $\ti{F} \su T$ Borel with  $\hal(T \setminus \ti{F})=0$. Moreover, we can extend the function $\phi$ arbitrarily to the whole space to get \eqref{eqpi13Zoli2}. 
For the converse, let $F \su  T$ with $\hal(T \setminus F)=0$ as in \eqref{eqpi13Zoli2}. Using that $\hal(T)<\infty$, choosing $U \su A$ Borel with $\hal(A \setminus U)=0$ and taking $\ti{F}=F \cap U$, we have that $\ti{F} \su A$ Borel with $\hal(\ti{F})=\hal(A)$.  Restricting $\phi$ to this set yields \eqref{eqpi13Zoli1}.

	Now we introduce the concept of formulas, and will show that \eqref{eqpi13Zoli2} is equivalent to a so called $\Pi^1_3$ formula. 	Let $\mathcal{X}$ be the set of \emph{recursively presented} Polish spaces with a fixed presentation, which in particular includes a fixed enumeration of basic open sets of the space. 
This $\mathcal{X}$ naturally includes all the spaces $\N,2^\N,\R^k$, and finite products and spaces of compact subsets (equipped with the Hausdorff topology) of all of those (see, \cite[3.B]{Mo}). 
	
	A formula is called $\Sigma^1_n$ if it has the form
	\[\underbrace{\exists x_1 \in X_1 \forall x_2 \in X_2 \dots}_{\text{$n$-many}} \psi(x_1,\dots,x_n),\]
	where $X_i \in \mathcal{X}$ are uncountable spaces, and $\psi$ uses only quantifiers over $\N$, recursive functions between powers of $\N$, and the enumeration of the basic open sets of spaces of $\mathcal{X}$, together with the $\epsilon$-relation and the usual logical operations. 
The $\Pi^1_n$ formulas are defined by formulas similar to the ones above, but starting with a universal quantifier instead of the existential quantifier. In particular, $\Pi^1_3$ formulas are of the form: 
$$\forall x_1 \in X_1 \exists x_2 \in X_2 \forall x_3 \in X_3 \ \psi(x_1,x_2,x_3).$$
	
	It is easy to see that using multiple quantifiers of the same sort consecutively, or adding quantifications over $\N$ do not change the above classes. Note also that this notion coincides with the usual definition of $2$nd order formulas over the structure $(\N,+,\cdot)$. We will say that a set $A \subset X$ for some $X \in \mathcal{X}$ is $\Sigma^1_n$ (resp $\Pi^1_n$), if its definable by such a formula. Using this notation, analytic sets are $\Sigma^1_1$, co-analytic sets are $\Pi^1_1$, and projections of co-analytic sets are $\Sigma^1_2$. 
\begin{lemma}
	\label{l:shoenfield}
	Let $\varphi$ be a $\Pi^1_3$ formula such that $\varphi$ follows from the
	Lebesgue measurability of  $\mathbf{\Sigma}^1_2$ subsets of $[0,1]$.
	Then $\varphi$ follows from ZFC.
	
\end{lemma}	
\begin{proof}
 	By Shoenfield's Absoluteness Theorem (\cite[Theorem 25.20]{Je}), $\Pi^1_3$ formulas are downward absolute. Consequently, if such a formula holds in a forcing extension of a model of every large enough fragment of ZFC, then it must hold in the original model as well.
	
	It is well known (see e.g. \cite[Corollary 9.3.2]{BJ}) 
that ${\rm add}({\mathcal{N}})>\aleph_1$ 
(which means that the union of $\aleph_1$ sets of Lebesgue
measure zero must have Lebesgue measure zero) implies that 
every $\mathbf{\Sigma}^1_2$ subset of $\R$ is Lebesgue measurable.	  	
 	
 	It is also a folklore statement of set theory (and can be deduced
 	for example from \cite[Model 7.6.1]{BJ} and \cite[Theorem 6.17]{Ku}) that
 	every  model of some large enough fragment of ZFC admits a forcing extension where 
${\rm add}({\mathcal{N}})>\aleph_1$.
Hence $\varphi$ must hold in each model of a large enough fragment of ZFC. So it follows from ZFC by the compactness theorem.
\end{proof}

	In order to establish \eqref{eqpi13Zoli2} in ZFC, we will show the following. 
\begin{lemma}
\label{l:pi13lem}
\eqref{eqpi13Zoli2} is equivalent to a $\Pi^1_3$ formula. 
\end{lemma}

Then, recalling that by Proposition \ref{*prop}, \eqref{eqpi13Zoli2} holds if all $\Sigma^1_2$ subsets of 
$[0,1]$ are Lebesgue measurable,
combining this with Lemma \ref{l:shoenfield} and \ref{l:pi13lem} yields \eqref{eqpi13Zoli2} in ZFC and the proof of Lemma \ref{genmeas} concludes. 

	\begin{proof}
	To prove Lemma \ref{l:pi13lem}, we will use the concepts of Borel codes. Let us collect their basic properties (see \cite[3.H]{Mo}). For any $X \in \mathcal{X}$ there exist sets $\Borelcodes{X} \subset 2^\N$, $\Analytic{X},\Coanalytic{X} \subset 2^\N \times X$ satisfying the following:
		\begin{itemize}
			\item $\Analytic{X}$ is $\Sigma^1_1$, $\Borelcodes{X}$ and $\Coanalytic{X}$ are $\Pi^1_1$,
			\item for each Borel set $B \subset X$ there exists a $c \in \Borelcodes{X}$ with $B=\Analytic{X}_c=\Coanalytic{X}_c$,
			\item for every $c \in \Borelcodes{X}$ we have $\Analytic{X}_c=\Coanalytic{X}_c$, and $\Analytic{X}_c$ is Borel, since it is both analytic and co-analytic.
		\end{itemize}
	Let us use the notation $\Borelcodes{k},\Analytic{k},\Coanalytic{k}$ instead of $\Borelcodes{\R^k},\Analytic{\R^k},\Coanalytic{\R^k}$, respectively.
	
	We will need an encoding of the Borel functions as well, denote by $\Borelfunctions{X,Y}$ the set of codes $c \in \Borelcodes{X \times Y}$ such that $\Analytic{X \times Y}_c$ is the graph of a total Borel function from $X$ to $Y$. 
	
	\begin{lemma}
		 \label{l:codesoffunctions} For any $X,Y \in \mathcal{X}$  the set $\Borelfunctions{X,Y}$ is $\Pi^1_1$. 
	\end{lemma}

 	Unfortunately, we were not able to locate a reference for this fact, so, for the sake of completeness we prove it here.
 	\begin{proof}
 		 Now, $c \in \Borelfunctions{X,Y}$ iff 
 		\[\forall x \ \forall z,z' \ \big(((x,z) \in \Analytic{X \times Y}_c \text{ and } (x,z') \in \Analytic{X\times Y}_c) \implies z=z'\big ),\]
 		and 
 		\[\forall x \ \exists z \ ((x,z) \in \Coanalytic{X\times Y}_c).\]
 		The former condition is clearly $\Pi^1_1$, while using the effective uniformization criterion \cite[4D.3]{Mo} the latter can be rewritten as 	\[\forall x \ \exists z \in \Delta^1_1(x,c) \ ((x,z) \in \Coanalytic{X \times Y}_c).\]
 		This also defines a $\Pi^1_1$ set by Kleene's theorem \cite[4D.4]{Mo} (for the definition of the class $\Delta^1_1$ see \cite[3.E]{Mo}).
 	\end{proof}
	
	Now consider the formula $\varphi'$:
	\[\forall t,b \in 2^\N \ \forall \epsilon>0\ \Big(t \not \in \Borelcodes{k} \text{ or } b \not \in \Borelcodes{k+n} \text{ or } \Analytic{k+n}_{b} \not \subset \Coanalytic{k}_t \times \R^n \text{ or } \]
	\[\Hausdorff{\alpha}(\Analytic{k}_t)=0 \text{ or } \Hausdorff{\alpha}(\Analytic{k}_t)=\infty \text{ or } \exists x \in \R^k \ \Hausdorff{\beta}((\Analytic{k+n}_{b})_x)>0 \text{ or }\]
	\[\exists f,\phi \in 2^\N \ \big(f \in \Borelcodes{k} \text{ and } \phi  \in \Borelfunctions{\R^k,2^\N}  \text{ and } \Hausdorff{\alpha}(\Analytic{k}_t \setminus \Coanalytic{k}_f)=0 \text{ and }\]
	\[  \forall x \in \R^k, y \in \R^n, z \in 2^\N \ (x \not \in \Analytic{k}_f \text{ or }  
	y \not \in (\Analytic{k+n}_{b})_x \text{ or } (x,z) \not \in \Analytic{\R^k \times 2^\N}_\phi \text{ or }\]
	\[
	 y \in 
	 \mathcal{U}(z) \text{ and }  \mathcal{S}(z) \leq \epsilon)\big )\Big),\]
	 
	 where we identify $H$ with $2^\N$ in the natural way.
	
	It should be clear from the properties of Borel codes that $\varphi'$ is equivalent to \eqref{eqpi13Zoli2}.

\begin{lemma}
	\label{l:pi13} The formula $\varphi'$ is equivalent to a $\Pi^1_3$ formula. 
\end{lemma}
\begin{proof}

	The conjunction of the following observations yields that $\varphi'$ has the desired complexity.
	\begin{enumerate}
	\item $t \not \in \Borelcodes{k}$ and $b \not \in \Borelcodes{k+n}$ are $\Sigma^1_1$.
\item $\Analytic{k+n}_{b} \not \subset \Coanalytic{k}_t \times \R^n$ is equivalent to $\exists x \ (x \in \Analytic{k+n}_{b} \land x \not \in \Coanalytic{k}_t \times \R^n)$, hence it is also $\Sigma^1_1$.
\item \label{c:hnull} $\Hausdorff{\alpha}(\Analytic{k}_t)=0$ is equivalent to $\forall n  \ \exists z \ \forall y$ \[ (y \not \in \Analytic{k}_t \text{ or } y \in \mathcal{U}(z)) \text{ and } \sum_{D \in z} diam^\alpha(D)<\frac{1}{n},\] by the definition of $\Hausdorff{\alpha}$. This is a $\Sigma^1_2$ formula.
\item \label{c:hinfty} $\Hausdorff{\alpha}(\Analytic{k}_t)=\infty$ is equivalent to $\forall n \ \exists K$ compact $\forall y \ ((y \not \in K$ or $y \in \Coanalytic{k}_t)$ and  $\Hausdorff{\alpha}(K)>n),$ by the regularity of $\Hausdorff{\alpha}$ on analytic sets and by $\Analytic{k}_t=\Coanalytic{k}_t$.  Since the collection of compacts with Hausdorff measure $>n$ is a $\Pi^1_1$ (and also $\Sigma^1_1$) set, this is a $\Sigma^1_2$ formula.

\item  $\Hausdorff{\beta}((\Coanalytic{k+n}_{b})_x)>0$ is equivalent to a $\Sigma^1_2$ formula, similarly to \ref{c:hinfty}.
\item $f \in \Borelcodes{k} \text{ and } \phi  \in \Borelfunctions{\R^k,2^\N}$ are $\Pi^1_1$ by Lemma \ref{l:codesoffunctions}.
\item $\Hausdorff{\alpha}(\Analytic{k}_t \setminus \Coanalytic{k}_f)=0$ is equivalent to a $\Sigma^1_2$ formula, as in \ref{c:hnull}.
\item $x \not \in \Analytic{k}_f,
y \not \in (\Analytic{k+n}_{b})_x, (x,z) \not \in \Analytic{\R^k \times 2^\N}_\phi$, $y \in 
\mathcal{U}(z)$, $\mathcal{S}(z) \leq \epsilon$ are $\Pi^1_1$.
		 
\end{enumerate}
	
\end{proof}
\end{proof}

\end{document}